\documentclass[11pt]{amsart}

\usepackage{amsmath,amssymb,amscd}
\usepackage{color}
\usepackage{hyperref}
\usepackage{mathrsfs}

\usepackage{xy}
\xyoption{all}

\newcommand{\nc}{\newcommand}

\def\Z{{\bf Z}}

\def\C{{\bf C}}

\def\A{{\bf A}}

\def\P{{\bf P}}

\nc{\CC}{{\mathbb{C}}}
\nc{\LL}{{\mathbb{L}}}
\nc{\RR}{{\mathbb{R}}}
\nc{\PP}{{\mathbb{P}}}
\nc{\OO}{{\mathbb{O}}}

\nc{\QQ}{{\mathbb{Q}}}
\nc{\ZZ}{{\mathbb{Z}}}

\nc{\cA}{{\mathscr{A}}}
\nc{\cB}{{\mathscr{B}}}
\nc{\cC}{{\mathscr{C}}}
\nc{\cD}{{\mathscr{D}}}
\nc{\cE}{{\mathscr{E}}}
\nc{\cF}{{\mathscr{F}}}
\nc{\cG}{{\mathscr{G}}}
\nc{\cH}{{\mathscr{H}}}
\nc{\cI}{{\mathscr{I}}}
\nc{\cJ}{{\mathscr{J}}}
\nc{\cK}{{\mathscr{K}}}
\nc{\cL}{{\mathscr{L}}}
\nc{\cM}{{\mathscr{M}}}
\nc{\cN}{{\mathscr{N}}}
\nc{\cO}{{\mathscr{O}}}
\nc{\cP}{{\mathscr{P}}}
\nc{\cQ}{{\mathscr{Q}}}
\nc{\bcQ}{{\overline{\mathscr{Q}}}}
\nc{\cR}{{\mathscr{R}}}
\nc{\cS}{{\mathscr{S}}}
\nc{\cT}{{\mathscr{T}}}
\nc{\cU}{{\mathscr{U}}}
\nc{\cV}{{\mathscr{V}}}
\nc{\cW}{{\mathscr{W}}}
\nc{\cX}{{\mathscr{X}}}
\nc{\cY}{{\mathscr{Y}}}
\nc{\cZ}{{\mathscr{Z}}}

\nc{\bA}{{\mathbf{A}}}
\nc{\bB}{{\mathbf{B}}}
\nc{\bC}{{\mathbf{C}}}
\nc{\bD}{{\mathbf{D}}}
\nc{\bE}{{\mathbf{E}}}
\nc{\bF}{{\mathbf{F}}}
\nc{\bG}{{\mathbf{G}}}
\nc{\bH}{{\mathbf{H}}}
\nc{\bI}{{\mathbf{I}}}
\nc{\bJ}{{\mathbf{J}}}
\nc{\bK}{{\mathbf{K}}}
\nc{\bL}{{\mathbf{L}}}
\nc{\bM}{{\mathbf{M}}}
\nc{\bN}{{\mathbf{N}}}
\nc{\bO}{{\mathbf{O}}}
\nc{\bP}{{\mathbf{P}}}
\nc{\bQ}{{\mathbf{Q}}}
\nc{\bR}{{\mathbf{R}}}
\nc{\bS}{{\mathbf{S}}}
\nc{\bT}{{\mathbf{T}}}
\nc{\bU}{{\mathbf{U}}}
\nc{\bV}{{\mathbf{V}}}
\nc{\bW}{{\mathbf{W}}}
\nc{\bX}{{\mathbf{X}}}
\nc{\bY}{{\mathbf{Y}}}
\nc{\bZ}{{\mathbf{Z}}}

\nc{\ba}{{\mathbf{a}}}
\nc{\bb}{{\mathbf{b}}}
\nc{\bc}{{\mathbf{c}}}
\nc{\bd}{{\mathbf{d}}}
\nc{\be}{{\mathbf{e}}}
\nc{\bg}{{\mathbf{g}}}
\nc{\bh}{{\mathbf{h}}}
\nc{\bi}{{\mathbf{i}}}
\nc{\bj}{{\mathbf{j}}}
\nc{\bk}{{\mathbf{k}}}
\nc{\bl}{{\mathbf{l}}}
\nc{\bm}{{\mathbf{m}}}
\nc{\bn}{{\mathbf{n}}}
\nc{\bo}{{\mathbf{o}}}
\nc{\bp}{{\mathbf{p}}}
\nc{\bq}{{\mathbf{q}}}
\nc{\br}{{\mathbf{r}}}
\nc{\bs}{{\mathbf{s}}}
\nc{\bt}{{\mathbf{t}}}
\nc{\bu}{{\mathbf{u}}}
\nc{\bv}{{\mathbf{v}}}
\nc{\bw}{{\mathbf{w}}}
\nc{\bx}{{\mathbf{x}}}
\nc{\by}{{\mathbf{y}}}
\nc{\bz}{{\mathbf{z}}}

\nc{\fA}{{\mathfrak{A}}}
\nc{\fB}{{\mathfrak{B}}}
\nc{\fC}{{\mathfrak{C}}}
\nc{\fD}{{\mathfrak{D}}}
\nc{\fE}{{\mathfrak{E}}}
\nc{\fF}{{\mathfrak{F}}}
\nc{\fG}{{\mathfrak{G}}}
\nc{\fH}{{\mathfrak{H}}}
\nc{\fI}{{\mathfrak{I}}}
\nc{\fJ}{{\mathfrak{J}}}
\nc{\fK}{{\mathfrak{K}}}
\nc{\fL}{{\mathfrak{L}}}
\nc{\fM}{{\mathfrak{M}}}
\nc{\fN}{{\mathfrak{N}}}
\nc{\fO}{{\mathfrak{O}}}
\nc{\fP}{{\mathfrak{P}}}
\nc{\fQ}{{\mathfrak{Q}}}
\nc{\fR}{{\mathfrak{R}}}
\nc{\fS}{{\mathfrak{S}}}
\nc{\fT}{{\mathfrak{T}}}
\nc{\fU}{{\mathfrak{U}}}
\nc{\fV}{{\mathfrak{V}}}
\nc{\fW}{{\mathfrak{W}}}
\nc{\fX}{{\mathfrak{X}}}
\nc{\fY}{{\mathfrak{Y}}}
\nc{\fZ}{{\mathfrak{Z}}}

\nc{\fa}{{\mathfrak{a}}}
\nc{\fb}{{\mathfrak{b}}}
\nc{\fc}{{\mathfrak{c}}}
\nc{\fd}{{\mathfrak{d}}}
\nc{\fe}{{\mathfrak{e}}}
\nc{\ff}{{\mathfrak{f}}}
\nc{\fg}{{\mathfrak{g}}}
\nc{\fh}{{\mathfrak{h}}}
\nc{\fj}{{\mathfrak{j}}}
\nc{\fk}{{\mathfrak{k}}}
\nc{\fl}{{\mathfrak{l}}}
\nc{\fm}{{\mathfrak{m}}}
\nc{\fn}{{\mathfrak{n}}}
\nc{\fo}{{\mathfrak{o}}}
\nc{\fp}{{\mathfrak{p}}}
\nc{\fq}{{\mathfrak{q}}}
\nc{\fr}{{\mathfrak{r}}}
\nc{\fs}{{\mathfrak{s}}}
\nc{\ft}{{\mathfrak{t}}}
\nc{\fu}{{\mathfrak{u}}}
\nc{\fv}{{\mathfrak{v}}}
\nc{\fw}{{\mathfrak{w}}}
\nc{\fx}{{\mathfrak{x}}}
\nc{\fy}{{\mathfrak{y}}}
\nc{\fz}{{\mathfrak{z}}}

\nc{\sA}{{\mathsf{A}}}
\nc{\sB}{{\mathsf{B}}}
\nc{\sC}{{\mathsf{C}}}
\nc{\sD}{{\mathsf{D}}}
\nc{\sE}{{\mathsf{E}}}
\nc{\sF}{{\mathsf{F}}}
\nc{\sG}{{\mathsf{G}}}
\nc{\sH}{{\mathsf{H}}}
\nc{\sI}{{\mathsf{I}}}
\nc{\sJ}{{\mathsf{J}}}
\nc{\sK}{{\mathsf{K}}}
\nc{\sL}{{\mathsf{L}}}
\nc{\sM}{{\mathsf{M}}}
\nc{\sN}{{\mathsf{N}}}
\nc{\sO}{{\mathsf{O}}}
\nc{\sP}{{\mathsf{P}}}
\nc{\sQ}{{\mathsf{Q}}}
\nc{\sR}{{\mathsf{R}}}
\nc{\sS}{{\mathsf{S}}}
\nc{\sT}{{\mathsf{T}}}
\nc{\sU}{{\mathsf{U}}}
\nc{\sV}{{\mathsf{V}}}
\nc{\sW}{{\mathsf{W}}}
\nc{\sX}{{\mathsf{X}}}
\nc{\sY}{{\mathsf{Y}}}
\nc{\sZ}{{\mathsf{Z}}}

\nc{\sa}{{\mathsf{a}}}
\nc{\sd}{{\mathsf{d}}}
\nc{\se}{{\mathsf{e}}}
\nc{\sg}{{\mathsf{g}}}
\nc{\sh}{{\mathsf{h}}}
\nc{\si}{{\mathsf{i}}}
\nc{\sj}{{\mathsf{j}}}
\nc{\sk}{{\mathsf{k}}}
\nc{\sm}{{\mathsf{m}}}
\nc{\sn}{{\mathsf{n}}}
\nc{\so}{{\mathsf{o}}}
\nc{\sq}{{\mathsf{q}}}
\nc{\sr}{{\mathsf{r}}}
\nc{\st}{{\mathsf{t}}}
\nc{\su}{{\mathsf{u}}}
\nc{\sv}{{\mathsf{v}}}
\nc{\sw}{{\mathsf{w}}}
\nc{\sx}{{\mathsf{x}}}
\nc{\sy}{{\mathsf{y}}}
\nc{\sz}{{\mathsf{z}}}

\newcommand{\rc}{\mathrm{c}}

\nc{\oA}{{\overline{A}}}
\nc{\oB}{{\overline{B}}}
\nc{\oC}{{\overline{C}}}
\nc{\oD}{{\overline{D}}}
\nc{\oE}{{\overline{E}}}
\nc{\oF}{{\overline{F}}}
\nc{\oG}{{\overline{G}}}
\nc{\oH}{{\overline{H}}}
\nc{\oI}{{\overline{I}}}
\nc{\oJ}{{\overline{J}}}
\nc{\oK}{{\overline{K}}}
\nc{\oL}{{\overline{L}}}
\nc{\oM}{{\overline{M}}}
\nc{\oN}{{\overline{N}}}
\nc{\oO}{{\overline{O}}}
\nc{\oP}{{\overline{P}}}
\nc{\oQ}{{\overline{Q}}}
\nc{\oR}{{\overline{R}}}
\nc{\oS}{{\overline{S}}}
\nc{\oT}{{\overline{T}}}
\nc{\oU}{{\overline{U}}}
\nc{\oV}{{\overline{V}}}
\nc{\oW}{{\overline{W}}}
\nc{\oX}{{\overline{X}}}
\nc{\oY}{{\overline{Y}}}
\nc{\oZ}{{\overline{Z}}}

\nc{\oa}{{\overline{a}}}
\nc{\ob}{{\overline{b}}}
\nc{\oc}{{\overline{c}}}
\nc{\od}{{\overline{d}}}
\nc{\of}{{\overline{f}}}
\nc{\og}{{\overline{g}}}
\nc{\oh}{{\overline{h}}}
\nc{\oi}{{\overline{i}}}
\nc{\oj}{{\overline{j}}}
\nc{\ok}{{\overline{k}}}
\nc{\ol}{{\overline{l}}}
\nc{\om}{{\overline{m}}}
\nc{\on}{{\overline{n}}}
\nc{\oo}{{\overline{o}}}
\nc{\op}{{\overline{p}}}
\nc{\oq}{{\overline{q}}}
\nc{\os}{{\overline{s}}}
\nc{\ot}{{\overline{t}}}
\nc{\ou}{{\overline{u}}}
\nc{\ov}{{\overline{v}}}
\nc{\ow}{{\overline{w}}}
\nc{\ox}{{\overline{x}}}
\nc{\oy}{{\overline{y}}}
\nc{\oz}{{\overline{z}}}

\nc{\tA}{{\tilde{A}}}
\nc{\tB}{{\tilde{B}}}
\nc{\tC}{{\tilde{C}}}
\nc{\tD}{{\tilde{D}}}
\nc{\tE}{{\tilde{E}}}
\nc{\tF}{{\tilde{F}}}
\nc{\tG}{{\tilde{G}}}
\nc{\tH}{{\tilde{H}}}
\nc{\tI}{{\tilde{I}}}
\nc{\tJ}{{\tilde{J}}}
\nc{\tK}{{\tilde{K}}}
\nc{\tL}{{\tilde{L}}}
\nc{\tM}{{\tilde{M}}}
\nc{\tN}{{\tilde{N}}}
\nc{\tO}{{\tilde{O}}}
\nc{\tP}{{\tilde{P}}}
\nc{\tQ}{{\tilde{Q}}}
\nc{\tR}{{\tilde{R}}}
\nc{\tS}{{\widetilde{S}}}
\nc{\tT}{{\tilde{T}}}
\nc{\tU}{{\tilde{U}}}
\nc{\tV}{{\tilde{V}}}
\nc{\tW}{{\tilde{W}}}
\nc{\tX}{{\tilde{X}}}
\nc{\tY}{{\tilde{Y}}}
\nc{\tZ}{{\tilde{Z}}}

\nc{\ta}{{\tilde{a}}}
\nc{\tb}{{\tilde{b}}}
\nc{\tc}{{\tilde{c}}}
\nc{\td}{{\tilde{d}}}
\nc{\te}{{\tilde{e}}}
\nc{\tf}{{\tilde{f}}}
\nc{\tg}{{\tilde{g}}}
\nc{\ti}{{\tilde{i}}}
\nc{\tj}{{\tilde{j}}}
\nc{\tk}{{\tilde{k}}}
\nc{\tl}{{\tilde{l}}}
\nc{\tm}{{\tilde{m}}}
\nc{\tn}{{\tilde{n}}}
\nc{\tp}{{\tilde{p}}}
\nc{\tq}{{\tilde{q}}}
\nc{\tr}{{\tilde{r}}}
\nc{\ts}{{\tilde{s}}}
\nc{\tu}{{\tilde{u}}}
\nc{\tv}{{\tilde{v}}}
\nc{\tw}{{\tilde{w}}}
\nc{\tx}{{\tilde{x}}}
\nc{\ty}{{\tilde{y}}}
\nc{\tz}{{\tilde{z}}}

\nc{\hA}{{\hat{A}}}
\nc{\hB}{{\hat{B}}}
\nc{\hC}{{\hat{C}}}
\nc{\hD}{{\hat{D}}}
\nc{\hE}{{\hat{E}}}
\nc{\hF}{{\hat{F}}}
\nc{\hG}{{\hat{G}}}
\nc{\hH}{{\hat{H}}}
\nc{\hI}{{\hat{I}}}
\nc{\hJ}{{\hat{J}}}
\nc{\hK}{{\hat{K}}}
\nc{\hL}{{\hat{L}}}
\nc{\hM}{{\hat{M}}}
\nc{\hN}{{\hat{N}}}
\nc{\hO}{{\hat{O}}}
\nc{\hP}{{\hat{P}}}
\nc{\hQ}{{\hat{Q}}}
\nc{\hR}{{\hat{R}}}
\nc{\hS}{{\widehat{S}}}
\nc{\hT}{{\hat{T}}}
\nc{\hU}{{\hat{U}}}
\nc{\hV}{{\hat{V}}}
\nc{\hW}{{\hat{W}}}
\nc{\hX}{{\widehat{X}}}
\nc{\hY}{{\hat{Y}}}
\nc{\hZ}{{\hat{Z}}}

\nc{\ha}{{\hat{a}}}
\nc{\hb}{{\hat{b}}}
\nc{\hc}{{\hat{c}}}
\nc{\hd}{{\hat{d}}}
\nc{\he}{{\hat{e}}}
\nc{\hf}{{\hat{f}}}
\nc{\hg}{{\hat{g}}}
\nc{\hh}{{\hat{h}}}
\nc{\hi}{{\hat{i}}}
\nc{\hio}{{\hat{\iota}}}
\nc{\hj}{{\hat{\jmath}}}
\nc{\hk}{{\hat{k}}}
\nc{\hl}{{\hat{l}}}
\nc{\hm}{{\hat{m}}}
\nc{\hn}{{\hat{n}}}
\nc{\ho}{{\hat{o}}}
\nc{\hp}{{\hat{p}}}
\nc{\hq}{{\hat{q}}}
\nc{\hr}{{\hat{r}}}
\nc{\hs}{{\hat{s}}}
\nc{\hu}{{\hat{u}}}
\nc{\hv}{{\hat{v}}}
\nc{\hw}{{\hat{w}}}
\nc{\hx}{{\hat{x}}}
\nc{\hy}{{\hat{y}}}
\nc{\hz}{{\hat{z}}}

\nc{\eps}{\varepsilon}
\nc{\lan}{\big\langle}
\nc{\ran}{\big\rangle}
\nc{\kk}{{\Bbbk}}
\nc{\bcV}{{\overline{\cV}}}
\nc{\bcA}{{\overline{\cA}}}
\nc{\tcA}{{\widetilde{\cA}}}
\nc{\tcV}{{\widetilde{\cV}}}

\nc{\SV}{\textnormal{Q}(V)}
\nc{\SK}{\textnormal{Q}(K)}

\def\isom{\simeq}
\def\hra{\hookrightarrow}

\DeclareMathOperator{\isomlra}{\stackrel{{}_{\scriptstyle\sim}}{\lra}}

\def\lhra{\ensuremath{\lhook\joinrel\relbar\joinrel\rightarrow}}

\def\setminus{\smallsetminus}
\def\cong{\isom}
\nc{\Ap}{{A^\perp}}

\renewcommand{\Im}{\operatorname{Im}}

\def\bw#1#2{\textstyle{\bigwedge\hskip-0.9mm^{#1}}\hskip0.2mm{#2}}

\nc{\hhX}{\widehat{\widehat{X\,}}}

\DeclareMathOperator{\Hom}{\mathrm{Hom}}

\DeclareMathOperator{\CHom}{\mathscr{H}\!\mathit{om}}

\DeclareMathOperator{\Hilb}{\mathrm{Hilb}}
\DeclareMathOperator{\Spec}{\mathrm{Spec}}

\DeclareMathOperator{\Sing}{\mathrm{Sing}}

\DeclareMathOperator{\Cl}{\mathrm{Cl}}

\DeclareMathOperator{\Ker}{\mathrm{Ker}}
\DeclareMathOperator{\Coker}{\mathrm{Coker}}

\DeclareMathOperator{\Gr}{\mathrm{Gr}}

\DeclareMathOperator{\LGr}{\mathrm{LGr}}

\DeclareMathOperator{\id}{\mathrm{id}}

\DeclareMathOperator{\codim}{\mathrm{codim}}

\DeclareMathOperator{\Sym}{\mathrm{Sym}}

\def\lin{\mathrel{\underset{\textnormal{lin}}{\equiv}}}
\def\notlin{\mathrel{\underset{\textnormal{lin}}{\not\equiv}}}

\def\lra{\longrightarrow}
\def\llra{\hbox to 10mm{\rightarrowfill}}
\def\lllra{\hbox to 15mm{\rightarrowfill}}

\theoremstyle{plain}

\newtheorem{theorem}{Theorem}[section]

\newtheorem{lemma}[theorem]{Lemma}

\newtheorem{proposition}[theorem]{Proposition}

\newtheorem{corollary}[theorem]{Corollary}

\theoremstyle{definition}

\newtheorem{definition}[theorem]{Definition}
\newtheorem{defi}[theorem]{Definition}

\theoremstyle{remark}

\newtheorem{remark}[theorem]{Remark}

\title[Double covers]{Double covers of quadratic degeneracy\\[1ex]and Lagrangian intersection loci}

  \author[O. Debarre]{Olivier Debarre}
\address{Universit\'e Paris-Diderot, PSL Research University, CNRS, 
\'Ecole normale sup\'erieure, D\'epartement de  Math\'ematiques et Applications,
45 rue d'Ulm, 75230 Paris cedex 05, France}
\email{{\tt olivier.debarre@ens.fr}}
 \author[A. Kuznetsov]{Alexander Kuznetsov}
 \address{Algebraic Geometry Section, Steklov Mathematical Institute,
  8 Gubkin str., Moscow 119991 Russia
 \newline
 Interdisciplinary Scientific Center J.-V.\ Poncelet, Independent University of Moscow, Russia
 \newline
 Laboratory of Algebraic Geometry, HSE University, Russia}
 \email{{\tt  akuznet@mi-ras.ru}}
 
\thanks{A.K. was partially supported by the HSE University Basic Research Program, Russian Academic Excellence Project~\mbox{``5--100''}
and by the Program of the Presidium of the Russian Academy of Sciences~01 ``Fundamental Mathematics and its Applications'' under grant PRAS-18-01.}


\begin{document}

\begin{abstract}
 We explain a general construction of  double covers of quadratic degeneracy loci and Lagrangian intersection loci
based on reflexive sheaves.\ We relate the double covers of quadratic degeneracy loci to the Stein factorizations 
of the relative Hilbert schemes of linear spaces of the corresponding quadric fibrations.\ We give a criterion for these double covers to be nonsingular.\ 

 {These results are an extension of  O'Grady's construction of  double covers  of EPW sextics 
and provide an alternate construction of  Iliev--Kapustka--Kapustka--Ranestad's EPW cubes.}
\end{abstract}

\maketitle

\section{Introduction}

When double coverings are mentioned,  one usually thinks of  double coverings branched over  divisors.\ 
 {These}  are very classical  objects in algebraic geometry.\ 
Let~$D$ be an effective Cartier divisor on a scheme $S$ such that the line bundle $\cO_S(D)$  
is a square in the Picard group of $S$, that is, 
\begin{equation*}
\cO_S(-D) \cong \cM^{\otimes2} 
\end{equation*}
for some line bundle $\cM$.\  The double covering $\tS$ of $S$ branched over $D$ is defined as the relative spectrum
\begin{equation*}
\tS := \Spec_S(\cO_S \oplus \cM),
\end{equation*}
where the algebra structure on $\cO_S \oplus \cM$ is such that the multiplication on $\cM$ is given by the composition 
\begin{equation*}
\cM \otimes \cM \cong \cO_S(-D) \xrightarrow{\ \cdot s_D\ } \cO_S
\end{equation*}
(here $s_D$ is a section of $\cO_S(D)$ with divisor $D$).\ 
This construction depends on the choice of the line bundle~$\cM$ (there may be several choices if the Picard group has nontrivial 2-torsion) 
and of the section~$s_D$ (the various choices form a torsor over the group $H^0(S,\cO_S^\times)$
of invertible functions on~$S$, and two choices provide  isomorphic double coverings if and only if
their ratio is a square of an invertible function).\ 
The  construction also works for the zero divisor $D$ and produces an \'etale double covering of~$S$ 
(which is nontrivial if the square root $\cM$ of~$\cO_S(-D) = \cO_S$ is nontrivial, and with the same ambiguity for the choice).

There is an extension of this construction which, although quite standard in birational geometry, is much less known.\ 
It is defined in a similar way, with the line bundle~$\cM$ replaced by a \emph{reflexive rank-$1$ sheaf} $\cR$.\ 
One considers the sum $\cO_S \oplus \cR$ as a sheaf of  {commutative} $\cO_S$-algebras, 
where the multiplication on the second summand is given by a  {symmetric} map $\bm \colon \cR \otimes \cR \to \cO_S$.\ 
Such a map automatically factors through the  {canonical map~$\cR \otimes \cR \to (\cR \otimes \cR)^{\vee\vee}$ to the reflexive hull}.\ 
We will concentrate on the special case  where
\begin{equation*}
(\cR \otimes \cR)^{\vee\vee} \cong \cO_S
\end{equation*}
(the reflexive sheaf $\cR$ is then \emph{self-dual})
and choose a map $\bm$ that factors through such an isomorphism  {(it is then automatically symmetric)}.\ 
The double covering 
\begin{equation*}
\tS := \Spec_S(\cO_S \oplus \cR) \lra S
\end{equation*}
is then \'etale outside of a subset of codimension~2 and (when all invertible functions are squares)
is canonically defined by the reflexive sheaf $\cR$.\ So, if there is a natural source of self-dual rank-1 reflexive sheaves, one obtains double covers as above.

In this paper, we apply   this principle in two situations.\ The first one concerns quadratic degeneracy loci: start from a family of quadratic forms, that is, from a morphism 
\begin{equation*}
q \colon \cL \lra \Sym^2(\cE^\vee)
\end{equation*}
from a line bundle $\cL$ to the symmetric square of a rank-$m$ vector bundle over a scheme $S$.\ Consider the subscheme $S_k \subset S$   where the morphism $\cL \otimes \cE \to \cE^\vee$ induced by $q$ has corank at least $k$.\ We define a reflexive rank-1 sheaf $\cR_k$ on $S_k$ by 
 \begin{equation*}
\cR_k := (\bw{k}\cC_k)^{\vee\vee},
\quad \textnormal{where}\quad
\cC_k = \Coker(\cL \otimes \cE \xrightarrow{\ q\ } \cE^\vee)\vert_{S_k}.
\end{equation*}
If $\cL^{\otimes(m-k)}\vert_{S_k} \cong \cM^{\otimes2}$  for some line bundle $\cM$ on $S_k$,
the determinant of the restricted quadratic form~$q\vert_{S_k}$ induces a self-duality isomorphism 
on the sheaf $\cM \otimes \cR_k \otimes \det(\cE)\vert_{S_k}$.\ Thus, we obtain a  double cover
\begin{equation*}
\tS_k := \Spec_{S_k}(\cO_{S_k} \oplus (\cM \otimes \cR_k \otimes \det(\cE)\vert_{S_k})) \lra S_k
\end{equation*}
of the quadratic degeneracy locus (see Theorem~\ref{theorem:covering-quadratic} for details).\ 
There is an ambiguity in the choice of the line bundle~$\cM$ if the  Picard group of $S_k$ has nontrivial 2-torsion, 
and in the choice of an isomorphism~$\cL^{\otimes(m-k)}\vert_{S_k} \cong \cM^{\otimes2}$,
but the local properties of the  double covers do not depend on  {these choices}.

The archetypical example of  this construction is the following.\ Let $V$ be a vector space of dimension $k + 1$ and let
$S = \Sym^2(V^\vee)$ be the space of quadratic forms on~$V$, equipped with the universal family 
$\cO_S \to \Sym^2(V^\vee \otimes \cO_S)$ of quadratic forms.\ The subscheme
$S_k\subset S$ is  the locus of rank-1 forms and the double cover constructed above is
 the morphism 
\begin{eqnarray*}
V^\vee &\lra& V^\vee/\pm1 \isom S_k
\\ 
\ell &\longmapsto& \ell^2.
\end{eqnarray*}
The scheme
$S_k$ is isomorphic to the cone over the double Veronese embedding of the projective space~$\P(V^\vee)$,
its ring of functions is isomorphic to the invariant ring~$\kk[V^\vee]_+$ of the involution $\ell \mapsto -\ell$
acting on $\kk[V^\vee]$, and the associated reflexive sheaf corresponds to the antiinvariant module $\kk[V^\vee]_-$ of the involution
(Lemma~\ref{lemma:veronese-cover}).

We show that under suitable assumptions (a \emph{regularity} property of the family of quadratic forms; 
see Definition~\ref{definition:p-regular}), the   double cover $\tS_k \to S_k$ can be obtained 
from the archetypical example by a smooth base change (away from the locus $S_{k+2}$).\ 
This provides a convenient  nonsingularity criterion for the double cover $\tS_k$ (Proposition~\ref{proposition:covering-smooth})
and allows us to describe the branch and the ramification loci of the cover.

We also relate the double cover $\tS_k \to S_k$ to the double covers obtained by the Stein factorization
of the   projections of the Hilbert schemes of linear isotropic spaces for the quadric fibration 
\begin{equation*}
\cQ \subset \P_S(\cE) \lra S
\end{equation*}
corresponding to the family of quadratic forms $q$: we prove in Proposition~\ref{proposition:isotropic-stein} that the two covers agree 
over the locus $S_k \setminus S_{k+1}$, so that $\tS_k$ is the normalization of the cover obtained from the Hilbert scheme 
(and if the Hilbert scheme is normal, the two covers are isomorphic).

In Section~\ref{subsection:symmetroids}, we apply our results to construct natural double covers 
of symmetroid hypersurfaces of odd degree (Theorem~\ref{thoerem:symmetroids}).

The second situation where our machinery works is the case of Lagrangian intersection loci.
Let~$\cV$ be a vector bundle  of rank~$2n$ over a scheme $S$, equipped with a family 
\begin{equation*}
\omega \colon \bw2\cV \lra \cL
\end{equation*}
of symplectic forms.\ Given a pair of Lagrangian subbundles $\cA_1,\cA_2 \subset \cV$, we consider the morphism
\begin{equation*}
\omega_{\cA_1,\cA_2} \colon \cA_1 \lhra \cV \xrightarrow[\sim]{\ \omega\ } \cV^\vee \otimes \cL \lra \cA_2^\vee \otimes \cL
\end{equation*}
and define a subscheme $S_k \subset S$ as the  corresponding corank-$k$ degeneration scheme.\
As in the quadratic case, we define a reflexive rank-1 sheaf $\cR_k$ on $S_k$ by 
\begin{equation*}
\cR_k = (\bw{k}\cC_k)^{\vee\vee},
\quad \textnormal{where}\quad
\cC_k = \Coker(\cA_1 \xrightarrow{\ \omega_{\cA_1,\cA_2}\ } \cA_2^\vee \otimes \cL)\vert_{S_k},
\end{equation*}
and note that if $(\cL^{\otimes(-n-k)} \otimes \det(\cA_1) \otimes {\det}(\cA_2))\vert_{S_k} \cong \cM^{\otimes 2}$
for some line bundle $\cM$ on $S_k$, we have a self-duality isomorphism on the sheaf $\cM \otimes \cR_k$.\
Thus, we obtain a double cover
\begin{equation*}
\tS_k = \Spec_{S_k}(\cO_{S_k} \oplus ( \cM \otimes \cR_k)) \lra S_k
\end{equation*}
 (see Theorem~\ref{theorem:covering-lagrangian} for details).\
We show that this double cover does not change under (appropriately defined) isotropic reduction 
(Proposition~\ref{proposition:isotropic-reduction}).\ We also check that \'etale locally, this double cover 
coincides with the double cover of a quadratic degeneracy locus for an appropriately defined
family of quadratic forms (Proposition~\ref{proposition:lag-quad}).\
This allows us to use  the  nonsingularity criterion developed in the quadratic situation (Corollary~\ref{corollary:lagrangian-smooth}).

In the last section, we provide applications of our results to EPW varieties  {$\sY_A^{\ge k} $,
$\sY_{A^\perp}^{\ge k}  $, and~$\sZ_A^{\ge k}$ defined below.}\ 
Let $V_6$ be a vector space of dimension~6.\ We endow the 20-dimensional vector space~$\bw3V_6$
with the  $\det(V_6)$-valued symplectic form defined by the wedge product.\ 
Given a Lagrangian subspace~$A \subset \bw3V_6$, one defines, for~$k\ge0$, three series of varieties
\begin{eqnarray*}
\sY_A^{\ge k} 		&=& \{ [v] \in \P(V_6) 		\mid \dim(A \cap (v \wedge \bw2V_6)) \ge k 	\},\\ 
\sY_{A^\perp}^{\ge k} 	&=& \{ [V_5] \in \P(V_6^\vee) 	\mid \dim(A \cap \bw3V_5) \ge k 		\},\\ 
\sZ_A^{\ge k}		&=& \{ [U_3] \in \Gr(3,V_6)	\mid \dim(A \cap (V_6 \wedge \bw2U_3)) \ge k 	\}.
\end{eqnarray*}
The first two were extensively investigated by O'Grady (\cite{og1,og2,og3,og4,og5,og6,og7}; 
the second series reduces to the first upon replacing~$V_6$ with~$V_6^\vee$ and~$A$ with~$A^\perp$) 
and the third by Iliev--Kapustka--Kapustka--Ranestad (\cite{IKKR}).\ 
When $A$ is sufficiently general, O'Grady constructed a double cover~$\widetilde\sY_A \to \sY_A^{\ge 1}$ 
and Iliev--Kapustka--Kapustka--Ranestad a double cover $\widetilde\sZ_A \to \sZ_A^{\ge 2}$, 
where $\widetilde\sY_A$ and $\widetilde\sZ_A$ are hyperk\"ahler varieties of respective dimensions 4 and 6 
(called \emph{double EPW sextic} and \emph{EPW cube}).\ 
We show (Theorems~\ref{theorem:y-covers} and~\ref{theorem:z-covers}) that our construction produces, for each $k$,  double covers
\begin{equation*}
\widetilde\sY_A^{\ge k} \lra \sY_A^{\ge k},
\qquad 
\widetilde\sY_{A^\perp}^{\ge k} \lra \sY_{A^\perp}^{\ge k},
\qquad 
\widetilde\sZ_A^{\ge k} \lra \sZ_A^{\ge k},
\end{equation*}
that give double EPW sextics and EPW cubes as special cases.\ These double covers also appear in the theory of Gushel--Mukai varieties~(\cite{DK1,DK2}) 
and this was the original motivation for this work.\ In the relative situation (for the universal family of EPW varieties), 
a  similar double cover is the base for a generalized root stack construction 
in terms of which the moduli stack of Gushel--Mukai varieties is described in~\cite{DK3}.

Throughout the article, we fix a field $\kk$ of characteristic different from~2.\  All schemes are assumed to be of finite type over  $\kk$.\

We thank Nicolas Addington for  suggesting to include the symmetroid example (Section~\ref{subsection:symmetroids})
and the referee for useful comments.

\section{Reflexive sheaves and double covers}

We start with a brief reminder of the correspondence between self-dual reflexive sheaves, 
2-torsion classes in the Weil divisors class groups, and double covers unramified in codimension~1.

 A  connected normal scheme is integral and any normal scheme is a disjoint union of normal integral schemes (\cite[Tag~033H]{sp}).\ 
A scheme is {\sf nonsingular} if its local rings are regular.

Let $\cR$ be a coherent sheaf on  a normal
scheme $S$.\ Its dual sheaf $\cR^\vee$ is defined as 
\begin{equation*}
\cR^\vee := \CHom(\cR,\cO_S).
\end{equation*}
The coherent sheaf $\cR$   is called {\sf reflexive} (\cite{Ha}) if the canonical morphism
\begin{equation*}
\cR \lra \cR^{\vee\vee}:=(\cR^\vee)^\vee 
\end{equation*}
is an isomorphism.\ Any locally free sheaf is reflexive, and so is
the  sheaf $\cR^\vee$   {for any coherent sheaf  $\cR$} (\cite[Corollary~1.2]{Ha}).

We  say that the sheaf $\cR$ {\sf has rank $r$} if there is a dense open subscheme $S_0 \subset S$ 
such that the restriction of $\cR$ to $S_0$ is locally free of rank $r$.\ Its  {\sf locally free locus} is the maximal open subscheme~\mbox{$S_0 \subset S$} with this property.

Rank-1 reflexive sheaves on $S$ form a group for the operation
\begin{equation*}
(\cR_1, \cR_2) \longmapsto 
(\cR_1 \otimes \cR_2)^{\vee\vee}
\end{equation*}
and the inverse of $\cR$ is $\cR^\vee $.\ 
This group is isomorphic to the Picard group of the nonsingular locus of~$S$ (\cite[Tag~0AVT]{sp}).

A {\sf Weil divisor} $D$ on  {a normal scheme} $S$ is a finite formal linear combination of integral subschemes of $S$ of codimension~1.\ 
Its {\sf Cartier locus} is the maximal open subset $S_0 \subset S$ such that the restriction of~$D$ to~$S_0$ is Cartier,  {that is, locally principal}.\ 
The complement of $S_0$ in $S$ is contained in the singular locus of  $S$, so it has codimension at least~2.\ 
The following lemma is classical (see~\cite[Notation~1.2]{koko} or~\cite[Tag~0AVT]{sp}).

\begin{lemma}
\label{lemma:reflexive}
Let $S$ be a normal scheme.\ There is a group isomorphism  $D\mapsto \cO_S(D)$ between the group~$\Cl(S)$ of linear equivalence classes of Weil divisors on $S$  
and the group of rank-$1$ reflexive sheaves on~$S$.\ This isomorphism is  compatible with flat base change between normal schemes.\ If~\mbox{$D \subset S$} is an integral subscheme of codimension~$1$, we have
\begin{equation*}
\cO_S(D) = (\cI_D)^\vee,
\end{equation*}
the dual of the ideal sheaf of $D$ in $S$.\ Moreover, for any $D$, the locally free locus of $\cO_S(D)$ is the Cartier locus of~$D$.
\end{lemma}

The following result is a simple consequence of the lemma.

\begin{corollary}
\label{corollary:self-duality}
Let $S$ be a normal  scheme and let $D$ be a Weil divisor on $S$.\ Then $D$ is a $2$-torsion class,  that is, $2D \lin 0$, if and only if the corresponding reflexive sheaf $\cO_S(D)$ is self-dual.
\end{corollary}

\begin{proof}
The first condition is equivalent to $D \lin -D$ which, in view of the isomorphism 
\begin{equation*}
\cO_S(D)^\vee \cong \cO_S(-D), 
\end{equation*}
is equivalent to the self-duality of $\cO_S(D)$.
\end{proof}

We now study double covers, by which we mean the following.

\begin{defi}\label{dc}
A morphism $f \colon \tS \to S$ is a {\sf double cover} if it is finite of degree 2 (that is, $f_*\cO_\tS$ has rank 2), but is not necessary flat, 
and there is an involution $\tau \colon \tS \to \tS$  over $S$ such that $S \cong \tS/\tau$.
\end{defi}

\begin{remark}
If $\tS$ and $S$ are both normal, there is no need to require the existence of the involution~$\tau$.\
Indeed, $\tS$ is then just the integral closure of $S$ in a degree-2 extension of the  {ring} of rational functions on $S$ 
(any such extension is a Galois extension since the characteristic of the base field is different from~2)
and the Galois group of the extension acts regularly on $\tS$ and gives the involution.
 \end{remark}

The relation with the notions discussed above is the following.

\begin{proposition}
\label{proposition:divisor-cover}
Assume $D$ is a $2$-torsion Weil divisor class on a normal    scheme $S$.\  There is a double cover $f \colon \tS \to S$, with $\tS$  normal, that satisfies   the following two properties:
\begin{itemize}
\item[ \textnormal{(a)}] there is an isomorphism $f_*\cO_\tS \cong \cO_S \oplus \cO_S(D)$;
\item[ \textnormal{(b)}] the morphism $f$ is \'etale over the Cartier locus of $D$.
\end{itemize}
The set of isomorphism classes of double covers  {satisfying properties \textnormal{(a)} and \textnormal{(b)}} is a torsor over the group 
\begin{equation*}
H^0(S,\cO_S^\times)/H^0(S,\cO_S^\times)^2
\end{equation*}
of invertible functions on $S$ modulo squares, and all these covers are \'etale locally  {over $S$} isomorphic to each other.
\end{proposition}

 {In Lemma~\ref{lemma:cover-divisor}, we will show that any double cover of a normal scheme 
which is \'etale over its nonsingular locus is obtained by this construction.}

\begin{proof}
Set $\cR := \cO_S(D)$.\
By Lemma~\ref{lemma:reflexive}, there is an  isomorphism $(\cR \otimes \cR)^{\vee\vee} \isomlra \cO_S$.\
Choose such an isomorphism and consider the composition
\begin{equation}
\label{eq:mu}
\bm \colon \cR \otimes \cR \lra (\cR \otimes \cR)^{\vee\vee} \isomlra \cO_S
\end{equation}
of the canonical morphism with the chosen isomorphism.\ The composition is symmetric (with respect to the permutation of factors in $\cR \otimes \cR$),  hence turns $\cO_S \oplus \cR$ into a sheaf of commutative $\cO_S$-algebras.\ This allows us to set
\begin{equation*}
\tS := \Spec_S (\cO_S \oplus \cR),
\end{equation*}
with  canonical map $f \colon \tS \to S$.\ This map is finite since $\cO_S \oplus \cR$ is coherent, and its degree is 2, the rank of~$\cO_S \oplus \cR$.\ The automorphism of the algebra $\cO_S \oplus \cR$ acting trivially on $\cO_S$ and by $-1$ on $\cR$ induces an involution $\tau$ of $\tS$ over $S$.\ The invariant subalgebra in $\cO_S \oplus \cR$ for this automorphism is~$\cO_S$, hence we have $\tS/\tau \cong S$.\ Thus, $f$ is a double cover in the sense of Definition \ref{dc}.

On the Cartier locus of $D$, the sheaf $\cR$ is invertible and $\bm$ is an isomorphism, hence~$f$ is \'etale.\ 
In particular, $f$ is \'etale over the nonsingular locus of $S$, hence $\tS$ is nonsingular in codimension~1.\ 
Moreover, the sheaf $\cO_S \oplus \cR$ satisfies Serre's condition $\mathbf{S}_2$:
 for the first summand, this follows from the normality of~$S$, 
and for the second summand, this follows from the reflexivity of $\cR$ (\cite[Theorem~1.9]{Ha94}).\ 
Therefore, the scheme $\tS$ satisfies condition~$\mathbf{S}_2$, hence is normal by Serre's criterion.

The only ambiguity in the construction of $f$ is the choice of  the algebra structure on $\cO_S \oplus \cR$,
that is, of the multiplication morphism $\cR \otimes \cR \to \cO_S$.\ 
Since $\cO_S$ is reflexive, this morphism factors through a morphism $(\cR \otimes \cR)^{\vee\vee} \to \cO_S$.\ 
Since we assume $f$ to be \'etale over the locally free locus of~$\cR$, that is, on the complement of a subset of codimension~2,
  this map is an isomorphism.\ Furthermore, if we multiply this isomorphism by the square of an invertible function, 
the isomorphism class of the algebra $\cO_S \oplus \cR$ will not change.\ 
It follows that the isomorphism class of~$f$ corresponds to the choice of an element in~$H^0(S,\cO_S^\times)/H^0(S,\cO_S^\times)^2$.\
Since any invertible function is \'etale locally a square, all these covers are \'etale locally isomorphic.
 \end{proof}

There is also a converse statement.

\begin{lemma}
\label{lemma:cover-divisor}
Assume that $f \colon \tS \to S$ is a double cover of normal schemes  that is \'etale over the nonsingular locus of $S$.\ 
There exists a Weil divisor $D$ on $S$ such that $2D \lin 0$ and $f_*\cO_\tS \cong \cO_S \oplus \cO_S(D)$.\
 If $f$ is \'etale everywhere,   $D$ is a Cartier divisor.
 \end{lemma}

\begin{proof}
Let $\tau$ be the involution of the double cover.\ It induces an involution of $f_*\cO_\tS$ over $\cO_S$.\ Since we have $\tS/\tau \cong S$, the invariant part is $\cO_S$.\
So, denoting by $\cR$ the antiinvariant part, we have a direct sum decomposition
\begin{equation*}
f_*\cO_\tS \cong \cO_S \oplus \cR.
\end{equation*}
The sheaf $\cR$ satisfies Serre's condition $\mathbf{S}_2$ since $\tS$ is normal, hence it is reflexive by~\cite[Theorem~1.9]{Ha94} 
(since $S$ is normal).\ Since $f$ has degree 2,  {the rank of $\cR$} is 1.\ By Lemma~\ref{lemma:reflexive}, we can therefore write
\begin{equation*}
\cR \cong \cO_S(D)
\end{equation*}
for some Weil divisor class $D$ on $S$.\ Furthermore, over the nonsingular locus of $S$, the map $f$ is \'etale, hence~$\cR$ is locally free 
and the multiplication $ \bm  \colon \cR \otimes \cR \to \cO_S$ is surjective.\ 
It is therefore an isomorphism (over the nonsingular locus of $S$) and the normality of $S$ implies $2D \equiv 0$.

 {If $f$ is \'etale everywhere,   the sheaf $\cR$ is locally free everywhere, 
hence the corresponding divisor class $D$ is Cartier.}
\end{proof}

\begin{lemma} 
\label{lemma:int}
In the situation of Proposition~\textup{\ref{proposition:divisor-cover}}, assume that the scheme $S$ is moreover integral.\ If~$D \notlin 0$, the scheme $\tS$ is also integral.
\end{lemma}

\begin{proof}
If $\tS$ is not integral, there is a subscheme $S' \subset \tS$ such that the map $S' \hra \tS \to S$ is birational.\ Since it is also finite, and $S$ is normal, the map is an isomorphism.\ Therefore, there is a morphism of sheaves $\gamma \colon \cR \to \cO_S$ such that $(1,\gamma) \colon \cO_S \oplus \cR \to \cO_S$ 
is an algebra homomorphism.\ This is equivalent to the equality $\bm = \gamma \otimes \gamma$, where $\bm$ is the map  {from}~\eqref{eq:mu}.\ On the other hand, $\bm$ is an isomorphism over the locally free locus of $\cR$, hence so is $\gamma$.\ In particular, $\gamma$ is an isomorphism over the complement of a subset of codimension~2, 
hence is an isomorphism over the entire scheme $S$, since both sheaves $\cR$ and~$\cO_S$ are reflexive.\ Finally, $\cR \cong \cO_S$ means $D \lin 0$.
\end{proof}

 We finish this section with a discussion of the branch and ramification loci.

\begin{definition}
\label{definition:branch-ramification}
Let $f \colon \tS \to S$ be a double cover of normal schemes which is \'etale over the nonsingular locus of $S$.\ 
Let $\cR$ be the corresponding reflexive sheaf on $S$ and let $\bm \colon \cR \otimes \cR \to \cO_S$ be the multiplication map.\ 
The image of $\bm$ is a sheaf of ideals on $S$.\ 
We call the corresponding subscheme of~$S$ {\sf the branch locus of $f$} and denote it $B(f) \subset S$.\ 
Since $\bm$ is an isomorphism on the nonsingular locus of $S$, we have~$B(f) \subset \Sing(S)$, thus  $B(f)$ has codimension~2 or more.

The natural morphism $\cO_S \oplus \cR \to \cO_{B(f)}$ (restriction on the first summand, and zero on the second)
is a surjective $\cO_S$-algebra homomorphism, hence induces a closed embedding
\begin{equation*}
\iota \colon B(f) = \Spec_S(\cO_{B(f)}) \lra \Spec_S(\cO_S \oplus \cR) = \tS
\end{equation*}
such that $f \circ \iota = \id_{B(f)}$.\
We call the subscheme $R(f) = \iota(B(f)) \subset \tS$ {\sf the ramification locus of $f$}.
\end{definition}

\begin{lemma}
The scheme-theoretic preimage $f^{-1}(B(f)) \subset \tS$ of the branch locus of $f$ is a nonreduced subscheme 
that contains the ramification locus and is contained in its first order infinitesimal neighborhood.
\end{lemma}

\begin{proof}
The scheme-theoretic preimage of $B(f)$ is isomorphic to $\Spec_{B(f)}(\cO_{B(f)} \oplus \cR\vert_{B(f)})$.\ Since the multiplication map $\bm\vert_{B(f)}$ is zero, the summand $\cR\vert_{B(f)}$ is a square-zero ideal,
while the first summand~$\cO_{B(f)}$ corresponds to the structure sheaf of $R(f)$.\ Therefore, $f^{-1}(B(f))$ contains $R(f)$ and is contained in its first order infinitesimal neighborhood.
\end{proof}

\section{Double covers of quadratic degeneracy loci}\label{sec3}

Let $S$ be a scheme, let $\cE$ be a vector bundle of rank $m$ on $S$, and let $\cL$ be a line bundle on $S$.
Consider a family of quadratic forms
\begin{equation*}
q \colon \cL \lra \Sym^2(\cE^\vee) 
\end{equation*}
on $S$.
We denote also by $q$ the associated map $\cL \otimes \cE \to \cE^\vee$ and by $\cC$ its cokernel.\
We have an exact sequence
\begin{equation}\label{eq:def-c}
\cL \otimes \cE \xrightarrow{\ q\ } \cE^\vee \lra \cC \lra 0.
\end{equation}
For any nonnegative integer $k$, we let $S_k \subset S$ be the corank-$k$ {\sf degeneracy locus} of $q$ 
with its natural scheme structure defined by the minors of $q$ of size $m + 1 - k$ 
and we let $\cC_k$ be the restriction  of the sheaf $\cC$ to   $S_k $.
Its further restriction to the open subscheme
 \begin{equation*}
S_k^0 := S_k \setminus S_{k+1} 
\end{equation*}
  is locally free of rank~$k$.

\subsection{Double covers}
\label{subsection:covers-quadratic}

The main result of this section is the construction of natural double covers of the schemes~$S_k$.\ Assume that $S_k$ is normal and that $S_k^0$ is dense in $S_k$.\ We consider the rank-1 reflexive sheaf 
\begin{equation}
\label{eq:rk-ck}
\cR_k := (\bw{k}\cC_k)^{\vee\vee}
\end{equation} 
on $S_k$; it is {invertible} on $S_k^0$.

\begin{theorem}\label{theorem:covering-quadratic}
Assume that $S_k$ is normal and that $\codim_{S_k}(S_{k+1}) \ge 2$.\ For each line bundle $\cM$ on~$S_k$ such that
\begin{equation}  
\label{eq:cm-cl}
\cL ^{\otimes(m-k)}\vert_{S_k} \cong \cM^{\otimes 2},
\end{equation}
there is a double cover $f_\cM \colon \tS_k \to  S_k$,  with  $\tS_k$ normal, that satisfies the following two properties:
\begin{itemize}
\item[\textnormal{(a)}] there is an isomorphism
\begin{equation*}
 f_{\cM*}\cO_{\tS_k} \cong \cO_{S_k} \oplus (\cM \otimes \cR_k \otimes \det(\cE)\vert_{S_k}),
\end{equation*}
\item[\textnormal{(b)}] the morphism $f_\cM$ is \'etale over the dense open subset $S_k^0 \subset S_k$.\
\end{itemize}
If all invertible functions on $S_k$ are  squares, such a double cover is unique up to isomorphism.
\end{theorem}

\begin{proof}
The map $\cE^\vee\vert_{S_k^0} \to  \cC_k\vert_{S_k^0}$ is an epimorphism of vector bundles, 
so we may consider $\cC_k^\vee\vert_{S_k^0}$ as a subbundle of $\cE\vert_{S_k^0}$.\
Set
\begin{equation}
\label{eq:cek}
\cE_k := (\cE\vert_{S_k^0})/(\cC_k^\vee\vert_{S_k^0}) 
\qquad{\textnormal{and}}\qquad 
\cL_k := \cL\vert_{S_k^0}.
\end{equation}
Note that $\cE_k$ is a vector bundle of rank $m - k$ on $S_k^0$ with a canonical isomorphism
\begin{equation}
\label{eq:det-cek}
\det(\cE_k) \cong (\cR_k \otimes \det(\cE))\vert_{S_k^0}.
\end{equation}
The quadratic form $q$ induces a canonical isomorphism
\begin{equation}
\label{eq:qk}
 q_k \colon \cL_k \otimes \cE_k \xrightarrow{\ \sim\ } \cE_k^\vee 
\end{equation}
of sheaves over ${S_k^0}$.\ Its determinant gives a canonical isomorphism
\begin{equation*}
 \det(q_k) \colon (\cL^{\otimes(m-k)} \otimes \cR_k \otimes \det(\cE))\vert_{S_k^0} \xrightarrow{\ \sim\ } (\cR_k \otimes \det(\cE))\vert_{S_k^0}^\vee.
\end{equation*}
In particular, for each square root $\cM$ of the line bundle $\cL^{\otimes(m-k)}\vert_{S_k} $, 
this gives a self-duality isomorphism for the line bundle $(\cM  \otimes \cR_k \otimes \det(\cE))\vert_{S_k^0}$.
Since $\codim_{S_k}(S_{k+1}) \ge 2$, it extends uniquely to a self-duality 
on the rank-1 reflexive sheaf $\cM  \otimes \cR_k \otimes \det(\cE)\vert_{S_k}$ on $S_k$.
By Proposition~\ref{proposition:divisor-cover}, this gives a double cover~$f_\cM \colon \tS_k \to S_k$,  with $\tS_k$ normal, 
that satisfies properties (a) and (b).\ The uniqueness also follows from Proposition~\ref{proposition:divisor-cover}.
 \end{proof}

\begin{remark}\label{remark:unique-cover-quadratic}
 If $m - k$ is even, there is a natural choice for the line bundle $\cM$ in~\eqref{eq:cm-cl}, namely 
\begin{equation*}
\cM := \cL^{\otimes((m-k)/2)}\vert_{S_k}.
\end{equation*}
Similarly,  if  $\cL = \cO_S$, there is also, for any $m - k$, a canonical choice
\begin{equation*}
\cM := \cO_{S_k}.
\end{equation*}
In any of these situations, with this choice of $\cM$, the double cover of Theorem~\ref{theorem:covering-quadratic} becomes completely canonical.
We will refer to it as the {\sf canonical double cover} of the quadratic degeneracy locus $S_k$.
\end{remark}

When $k$ is even, we can relate the reflexive sheaf $\cR_k$ (and hence the sheaf giving the double cover) 
to the canonical class of $S_k$.\
For a coherent sheaf $\cF$ on a normal scheme $S$, we denote by $\rc_1(\cF)$ the Weil divisor class on $S$ 
corresponding to the Cartier divisor class $\rc_1({\det(\cF\vert_{S^0})})$ on the nonsingular locus $S^0  $ of $ S$.

\begin{lemma}
\label{lemma:ksk-quadratic} 
Assume that $S_k$ is normal, $\codim_S(S_k) = k(k+1)/2$, and $\codim_{S_k}(S_{k+1}) \ge 2$.\ 
There is {an equality} 
of Weil divisor classes
\begin{equation*}
K_{S_k} = 
\begin{cases}
K_S\vert_{S_k} + \rc_1(\cR_k) - k\rc_1(\cE) - \tfrac{k(m-k)}2 \rc_1(\cL) & \text{if $k$ is even,}\\
K_S\vert_{S_k} - (k+1)\rc_1(\cE) - \tfrac{(k+1)(m-k)}2 \rc_1(\cL) & \text{if $k$ is odd.}
\end{cases}
\end{equation*}
\end{lemma}

\begin{proof}
It is easy to see that the normal bundle to $S_k^0$ in $S$ is isomorphic to the sheaf~$\Sym^2(\cC_k)$.\
Therefore, we have an equality of Cartier divisor classes
\begin{equation*}
K_{S_k^0} = K_S\vert_{S_k^0} + \rc_1(\Sym^2(\cC_k)) = K_S\vert_{S_k^0} + (k+1)\rc_1(\cC_k) = K_S\vert_{S_k^0} + (k+1)\rc_1(\cR_k).
\end{equation*}
The proof of Theorem~\ref{theorem:covering-quadratic} shows that on $S_k^0$, we have {an equality} 
\begin{equation*}
2\rc_1(\cR_k) + 2\rc_1(\cE) + (m-k)\rc_1(\cL) = 0.
\end{equation*}
When $k$ is odd, we can use this to rewrite~$(k+1)\rc_1(\cR_k)$ as a linear combination of~$\rc_1(\cE)$ and~$\rc_1(\cL)$,
and when $k$ is even, we do the same for~$k\rc_1(\cR_k)$.\
Then we treat the equality of Cartier divisor classes obtained on~$S_k^0$ as an equality of Weil divisor classes on~$S_k$,
which gives the required formulas.
\end{proof}

\subsection{Smoothness criteria}

Let us first discuss a prototypical example of a double cover.

Let $k$ be a positive integer,
let $V$ be a $\kk$-vector space of dimension $k+1$, and let 
\begin{equation*}
\SV := \Sym^2(V^\vee)  \isom \A_\kk^{(k+1)(k+2)/2}
\end{equation*}
be the affine space of all quadratic forms on~$V$.\
Consider the trivial bundles 
\begin{equation*}
\cE := V \otimes \cO_{\SV}
\qquad\text{and}\qquad
\cL := \cO_{\SV}, 
\end{equation*}
and let $q \colon \cL \to \Sym^2(\cE^\vee)$ be the universal quadratic form.\ The corresponding $k$-th degeneration scheme~$\SV_k \subset \Sym^2(V^\vee)$ is the affine cone 
over the double Veronese embedding $\P(V^\vee) \subset \P(\SV)$ and the~$(k+1)$-th degeneration scheme is the single point  $\SV_{k+1} = \{0\}$,
   both   with the reduced scheme structure.\ Since the line bundle $\cL$ is trivial, Theorem~\ref{theorem:covering-quadratic} provides a canonical double cover of~$\SV_k$.

\begin{lemma}
\label{lemma:veronese-cover}
The map $V^\vee \to   \SV$ that takes a linear form to its square factors as
\begin{equation}
\label{eq:veronese-map}
V^\vee \xrightarrow{\ f_0\ } \SV_k \lhra \SV,
\end{equation}
where $f_0$ is a double cover branched over the vertex of the cone.\ Moreover,~$f_0$ agrees with the canonical double cover of $\SV_k$ 
provided by Theorem~\textup{\ref{theorem:covering-quadratic}}.\

The branch locus of $f_0$ is the reduced point $\{0\} \in \SV$, 
its ramification locus is the reduced point $\{0\} \in V^\vee$,
and the preimage of the branch locus is the first order infinitesimal neighborhood of the ramification locus.
\end{lemma}

\begin{proof}
 The factorization through $\SV_k$ is clear and the map $f_0$ is a double cover because it is the quotient for the $\Z/2$-action on $V^\vee$ defined by $\ell \mapsto -\ell$.
 
 Denote by $\kk[V^\vee]_+$ and $\kk[V^\vee]_-$ the invariant and   antiinvariant parts of~$\kk[V^\vee]$
with respect to this $\Z/2$-action, so that 
\begin{equation*}
\kk[V^\vee] = \kk[V^\vee]_+ \oplus \kk[V^\vee]_-.
\end{equation*}
Then $\SV_k \cong \Spec(\kk[V^\vee]_+)$, {and} under the identification of sheaves on $\SV_k$ with $\kk[V^\vee]_+$-modules,
this direct sum provides the decomposition of $f_{0*}\cO_{V^\vee}$ into   invariant and antiinvariant parts.\
 {As explained in the proof of Lemma~\ref{lemma:cover-divisor}}, the reflexive sheaf $\cR $ corresponding to the double cover $f_0$ 
is the sheaf associated with~$\kk[V^\vee]_-$,  considered as a $\kk[V^\vee]_+$-module.

 {Consider the natural   morphisms 
\begin{equation*}
(V \otimes \kk[V^\vee]_+) \otimes (V \otimes \kk[V^\vee]_+) \to \kk[V^\vee]_- \otimes \kk[V^\vee]_- \to \kk[V^\vee]_+ 
\end{equation*}
of $\kk[V^\vee]_+$-modules, 
where both maps are induced by the multiplication inside $\kk[V^\vee]$ (note that $V$ is the space of linear functions on~$V^\vee$).\ Their composition coincides with the restriction of $q$ to $\SV_k$.\ Therefore, the map $q\vert_{\SV_k} \colon \cE\vert_{\SV_k} \to \cE\vert_{\SV_k}^\vee$ can be rewritten as the composition
\begin{equation*}
V \otimes \kk[V^\vee]_+ \to \kk[V^\vee]_- \to V^\vee \otimes \kk[V^\vee]_+,
\end{equation*}
where the first map is the multiplication and the second is its transposed.\ 
 {Therefore, over the open subset} $\SV_k^0=\SV_k\setminus\{0\}$, 
the sheaf $\cE_k = \Im(q\vert_{\SV_k^0})$ can be identified with $\kk[V^\vee]_-$  
and the induced family of quadratic forms   is given by the multiplication in $\kk[V^\vee]_-$ with values in $\kk[V^\vee]_+$.\ 
This means that the algebra structure on~$\cO \oplus \cR_k$ defined in Theorem~\ref{theorem:covering-quadratic} 
coincides over $\SV_k^0$ with the natural algebra structure on~\mbox{$\kk[V^\vee] = \kk[V^\vee]_+ \oplus \kk[V^\vee]_-$}.\ 
Thus, the double covers agree  over $\SV_k^0$, and since~$V^\vee$ is normal, they agree everywhere.}

Finally,  {by Definition~\ref{definition:branch-ramification},} the ideal of the branch locus $B(f_0)$ is the image of the multiplication map 
\begin{equation*}
\kk[V^\vee]_- \otimes \kk[V^\vee]_- \lra  \kk[V^\vee]_+.
\end{equation*}
It is generated by all monomials of degree~2 in the coordinates on $V^\vee$, that is, by the coordinates on $\SV$.\ Thus, $B(f_0)$ is the origin in $\SV$ with the reduced scheme structure.\ Consequently, the ramification locus $R(f_0)$ is the origin in $V^\vee$ with the reduced scheme structure as well.\ Since the map~$f_0$ is given by all monomials of degree~2, the scheme-theoretic preimage of the branch locus is the first order infinitesimal neighborhood of $R(f_0)=\{0\}$.
\end{proof}

 We will need the following definition.

\begin{defi}
\label{definition:p-regular}
A family of quadratic forms $q \colon \cL \to \Sym^2(\cE^\vee)$ over a scheme $S$ is called {\sf $p$-regular} 
at a point $s \in S$ if $s$ is nonsingular on $S$ and, for any subspace $K \subset \Ker(q_s) \subset \cE_s$ 
such that $\dim(K) \le p$, the canonical morphism
\begin{equation*}
dq \colon T_{S,s} \otimes \cL_s \to 
 \Sym^2(\cE_s^\vee) \to \Sym^2(K^\vee)
\end{equation*}
is surjective, where $T_{S,s}$ is the tangent space of $S$ at $s$.
 \end{defi}

\begin{lemma}
\label{lemma:regularity}
Let $s $ be a  point of $ S_k \setminus S_{k+1}$ nonsingular on $S$ and let $p$ be an integer such that $p\ge k $.\
Then $q$ is $p$-regular at $s$ if and only if $S_k$ is nonsingular of \textup(expected\textup) codimension $k(k+1)/2$ in~$S$ at~$s$.
\end{lemma}

\begin{proof}
Since $p\ge k =\dim(\Ker(q_s))$, the family of quadratic forms $q$ is $p$-regular  {at $s$}
if and only if the morphism $T_{S,s} \otimes \cL_s \to \Sym^2(\Ker(q_s)^\vee)$ is surjective.\
Since its kernel is $T_{S_k,s} \otimes \cL_s$,   the lemma follows.
\end{proof}

The following proposition is a very useful criterion for proving the  {nonsingularity} of the double covers 
associated with quadratic degeneracy loci.

\begin{proposition}\label{proposition:covering-smooth}
Assume that $S$ is nonsingular, that~$S_k$ is normal, that~$\codim_{S_k}(S_{k+1}) \ge 2$, that~$S_{k+2} $ is empty, 
and that $q$ is $(k+1)$-regular at all points of $S_{k}$.\

For any double cover $f \colon \tS_k \to S_k$ provided by Theorem~\textup{\ref{theorem:covering-quadratic}}, 
the scheme $ \tS_k$ is nonsingular, the branch locus of $f$ equals~$S_{k+1}$ as  schemes,
 and the preimage of the branch locus is the first order infinitesimal neighborhood of the ramification locus.
\end{proposition}

\begin{proof}
Since the map $f$ is \'etale over $S_k \setminus S_{k+1}$ and $S_k \setminus S_{k+1}$ is nonsingular by Lemma~\ref{lemma:regularity}, 
it is enough to verify that $\tS_k $ is  {nonsingular} over any point $s \in S_{k+1}$.\
Since the question is local on $S$, we may assume that the line bundle $\cL$ is trivial and that there is a $q$-orthogonal direct sum decomposition
\begin{equation*}
\cE = \cE' \oplus \cE'',
\end{equation*}
where $q\vert_{\cE'}$ is everywhere nondegenerate, $q'':=q\vert_{\cE''}$ vanishes at $s$, and $\cE''=V\otimes\cO_S$ is free of rank~$k+1$.\
The summand $q''$ defines a morphism
\begin{equation}
\label{eq:map-to-sv}
\psi \colon S \lra \Sym^2(V^\vee) = \SV 
\end{equation}
and one can view $q''$ as the pullback of the universal family of  {quadratic forms} on $\SV$.\ Note that the differential of $\psi$ is surjective at $s$ by $(k+1)$-regularity of $q$,
hence $\psi$ is smooth {in a neighborhood of~$s$}.\ Furthermore, 
 we have
 \begin{equation*}
S_k = \psi^{-1}(\SV_k) 
\qquad \textnormal{and}\qquad
S_{k+1} = \psi^{-1}(\{0\}),
\end{equation*}
 and the determinant cokernel sheaf $\cR_k$ (defined in \eqref{eq:rk-ck}) on $S_k$
 is the pullback of the determinant cokernel sheaf on $\SV_k$.\ 
 {By Lemma~\ref{lemma:veronese-cover},} the scheme
\begin{equation*}
\hS_k := V^\vee \times_{\SV_k} S_k,
\end{equation*}
where the fiber product is taken with respect to the maps~\eqref{eq:veronese-map} and~\eqref{eq:map-to-sv},
provides a double cover 
\begin{equation*}
\hat{f} \colon \hS_k \lra S_k
\end{equation*}
that satisfies {properties (a) and (b) in Theorem~\ref{theorem:covering-quadratic}}.\ 
Since $\psi$ is smooth {at $s$} and $V^\vee$ is nonsingular, it follows that $\hS_k$ is nonsingular at $\hat{f}^{-1}(s)$.\ 
Since all double covers~$f \colon \tS_k \to S_k$ provided by Theorem~\ref{theorem:covering-quadratic}
are \'etale locally isomorphic to each other, $\tS_k$  is also nonsingular at $f^{-1}(s)$.

Similarly, the branch locus of $f \colon \tS_k \to S_k$ is equal to the branch locus of $\hat{f}$,
which in   turn is equal to the pullback by $\psi$ of the branch locus of ${f_0 \colon} V^\vee \to \SV_k$.\ 
The latter is the point $\{0\}$ in $\SV$ (Lemma~\ref{lemma:veronese-cover}), 
hence the branch locus of $f$ is the zero-locus of $\psi$ which is, as we noted before, equal to~$S_{k+1}$.\ 
The preimage of the branch locus of ${\hat{f}}$ in $\hS_k$ equals the first order infinitesimal neighborhood 
of the ramification locus, hence the same is true for every double cover 
  {provided by Theorem~\ref{theorem:covering-quadratic}}.
\end{proof}

\subsection{Families of isotropic spaces}

Another way of constructing double covers of   quadratic degeneracy loci is the following.\ 
Let as before   $q \colon \cL \to \Sym^2(\cE^\vee)$ be a family of quadratic forms 
in a rank-$m$ vector bundle $\cE$ on a scheme $S$ and let $\cQ \subset \P_S(\cE)$ be the corresponding family of quadrics.\ 
 {Consider} the relative Hilbert scheme 
\begin{equation}\label{rhs}
\varphi\colon F_{d-1}(\cQ/S) := \Hilb^{\P^{d-1}}(\cQ/S)\lra S
\end{equation}
parameterizing projective linear spaces of dimension $d-1$ in the fibers of $\cQ$ over $S$.\
 {We make a couple of observations.}

\begin{lemma}
\label{lemma:f-cm}
Assume that $S$ is Cohen--Macaulay, that
\begin{equation}
\label{eq:dim-f}
\dim(F_{d-1}(\cQ/S)) = \dim(S) + d(m-d) - d(d+1)/2, 
\end{equation} 
and that $F_{d-1}(\cQ/S)$ is nonsingular in codimension~$1$.\ Then $F_{d-1}(\cQ/S)$ is normal.
\end{lemma}

\begin{proof}
Let $\cU$ be the tautological bundle on the Grassmannian $\Gr_S(d,\cE)$.\ The scheme $F_{d-1}(\cQ/S)$ is the zero-locus of a global section of the vector bundle $\cL^\vee \otimes \Sym^2(\cU^\vee)$ 
of rank~$d(d+1)/2$ on the Cohen--Macaulay  {scheme} $\Gr_S(d,\cE)$ of dimension $\dim(S) + d(m-d)$.\
If condition \eqref{eq:dim-f} is satisfied, $F_{d-1}(\cQ/S)$ is itself Cohen--Macaulay.\ 
By Serre's normality criterion, since $F_{d-1}(\cQ/S)$ is nonsingular in codimension~1, it is normal.
\end{proof}

\begin{lemma}
\label{lemma:normality-f}
Assume that $S$ is nonsingular and that $\cQ/S$ is $p$-regular.\ 
Over $S \setminus S_{p+1}$, the scheme $F_{d-1}(\cQ/S)$ is nonsingular and the dimension condition~\eqref{eq:dim-f} holds.\
In particular, if $\cQ/S$ is $p$-regular and 
\begin{equation}
\label{eq:dim-f-2}
\dim(F_{d-1}(\cQ/S) \times_S S_{p+1}) \le \dim(S) + d(m-d) - d(d+1)/2 - 2,
\end{equation}
the scheme $F_{d-1}(\cQ/S)$ is normal.
\end{lemma}
\begin{proof}
The proof is analogous to the proof of~\cite[Proposition~2.1]{K14}.\  
Let $s \in S$ be a point such that the corank of $q_s$ is equal to $k$, with $k \le p$.\
Let $U \subset \cE_s$ be a $q_s$-isotropic subspace of dimension $d$.\ 
To show that $F_{d-1}(\cQ/S) \subset \Gr_S(d,\cE)$ is nonsingular at $[U]$, it is enough to check that the natural map 
\begin{equation*}
dq \colon T_{S,s} \oplus \Hom(U,\cE_s/U) \lra \Sym^2(U^\vee)
\end{equation*}
is surjective.\
Let $K := U \cap \Ker(q_s)$.\ 
Since the cokernel of the restriction of the map $dq$ to the second summand is isomorphic to $\Sym^2(K^\vee)$, 
it remains to check the surjectivity of the map~\mbox{$T_{S,s} \to \Sym^2(K^\vee)$} obtained by   restricting $dq$ to the first summand.\ 
But this is precisely the map from the definition of $p$-regularity (see Definition~\ref{definition:p-regular}).\ 
Since $\dim(K) \le \dim(\Ker(q_s)) = k \le p$, it is surjective.

The above argument proves that over $S \setminus S_{p+1}$, the scheme $F_{d-1}(\cQ/S)$ is nonsingular and the equality~\eqref{eq:dim-f} holds.\
If the inequality~\eqref{eq:dim-f-2} also holds, the codimension of $F_{d-1}(\cQ/S) \times_S S_{p+1}$ in $F_{d-1}(\cQ/S)$ is at least~2,
hence the equality~\eqref{eq:dim-f} holds over $S$ and  $F_{d-1}(\cQ/S)$ is nonsingular in codimension~1.\ 
It is therefore normal by Lemma~\ref{lemma:f-cm}.
\end{proof}

 {Fix an integer $k\in\{0,\dots,m\}$ such that $m - k$ is even and set 
\begin{equation}\label{eq:d}
d := (m + k)/2.
\end{equation}
In the next proposition, we relate the double covering of the degeneracy locus $S_k$ 
constructed in Theorem~\ref{theorem:covering-quadratic} and Remark~\ref{remark:unique-cover-quadratic} to the Hilbert scheme~$F_{d-1}(\cQ/S)$.}

\begin{proposition}\label{proposition:isotropic-stein}
Assume that  $S_k$ is normal  and that $\codim_{S_k}(S_{k+1}) \ge 2$.
If $m - k$ is positive and even  and the integer $d$ is defined by~\eqref{eq:d}, 
the map $\varphi \colon F_{d-1}(\cQ/S) \to S$ factors as
\begin{equation*}
F_{d-1}(\cQ/S) \lra S'_k \lra S_k \lhra S,
\end{equation*}
where the first morphism has connected fibers and the second morphism is   finite.\
If~$f \colon \tS_k \to S_k$ is the canonical double cover provided 
by Theorem~\textup{\ref{theorem:covering-quadratic}} and Remark~\textup{\ref{remark:unique-cover-quadratic}}, we have
\begin{equation}
\label{eq:spk-tsk}
S'_k \times_{S_k} S_k^0 \cong \tS_k \times_{S_k} S_k^0,
\end{equation}
and $\tS_k$ is the normalization of $S'_k$.\
In particular, if $F_{d-1}(\cQ/S)$ is normal, we have $S'_k \cong \tS_k$.
\end{proposition}

\begin{proof}
The fiber of $F_{d-1}(\cQ/S)$ over a  
point $s \in S$ parameterizes $d$-dimensional $q_s$-isotropic vector subspaces 
in the $m$-dimensional vector space $\cE_s$.\
If the fiber is nonempty, the rank of $q_s$ does not exceed $2(m-d) = m - k$, hence its corank is at least $k$.
This proves that $\varphi$ factors through $S_k$.\
We define $S'_k$ by the Stein factorization $F_{d-1}(\cQ/S) \to S'_k \to S_k$.\
Let us show that it has all required properties.

We use the notation of the proof of Theorem~\ref{theorem:covering-quadratic}: in particular, the vector bundles $\cE_k$ and $\cL_k$ are defined by~\eqref{eq:cek} and the morphism $q_k$ is defined by~\eqref{eq:qk}.\
We denote by 
\begin{equation*}
\cQ_k \subset \P_{S_k^0}(\cE_k)
\end{equation*}
the family of nondegenerate quadrics given by the quadratic form $q_k$.\ 
Let us prove that there is an isomorphism
\begin{equation}
\label{eq:f-isomorphism}
F_{d-1}(\cQ/S) \times_S S_k^0 \cong F_{d-k-1}(\cQ_k/S_k^0)
\end{equation}
of schemes over $S_k^0$.

We first construct a map from the left side to the right side of~\eqref{eq:f-isomorphism}.\
Denote by $\varphi_0$ the natural projection $F_{d-1}(\cQ/S) \times_S S_k^0 \to S_k^0$.\
Let $\cU \subset \varphi_0^*(\cE)$ be the rank-$d$ tautological subbundle 
on the Hilbert scheme $F_{d-1}(\cQ/S) \times_S S_k^0 \subset \Gr_{S_k^0}(d,\cE)$.\
Since a $d$-dimensional isotropic subspace for a quadratic form of rank $m-k = 2(m-d)$ 
on a vector space of dimension $m$ contains the kernel of the form, 
there is an inclusion $\varphi_0^*(\cC_k^\vee\vert_{S_k^0}) \hra \cU$ and the quotient bundle
\begin{equation*}
 \cU_k := \cU/\varphi_0^*(\cC_k^\vee\vert_{S_k^0})
\end{equation*}
is a subbundle in $\varphi_0^*(\cE_k)$ of rank $d - k$ which is isotropic for $q_k$.\ 
Therefore, $\cU_k$ induces a map 
\begin{equation*}
F_{d-1}(\cQ/S) \times_S S_k^0 \lra F_{d-k-1}(\cQ_k/S_k^0).
\end{equation*}

{Conversely, let $\varphi_k$ be the natural projection $F_{d-k-1}(\cQ_k/S_k^0) \to S_k^0$ 
and let $\cU_k \subset \varphi_k^*(\cE_k)$ be the rank-$(d-k)$ tautological subbundle 
on the Hilbert scheme $F_{d-k-1}(\cQ_k/S_k^0) \subset \Gr_{S_k^0}(d-k,\cE_k)$.\
Denote by~$\cU \subset \varphi_k^*(\cE)$ the preimage of $\cU_k$ 
with respect to the natural projection $\varphi_k^*(\cE) \to \varphi_k^*(\cE_k)$.\
By construction, $\cU$ is a rank-$d$ subbundle in $\varphi_k^*(\cE)$ which is isotropic for $q$.\
Therefore, it induces a map 
\begin{equation*}
F_{d-k-1}(\cQ_k/S_k^0) \lra F_{d-1}(\cQ/S) \times_S S_k^0.
\end{equation*}
The two constructed maps are clearly mutually inverse, and this proves~\eqref{eq:f-isomorphism}.

Since the Stein factorization is compatible with base changes, it follows from~\eqref{eq:f-isomorphism} that $S'_k \times_S S_k^0$
provides the Stein factorization for the map $\varphi_k \colon {F_{d-k-1}}(\cQ_k/S_k^0) \to S_k^0$.\
But $\cQ_k \to S_k^0$ is a family of nondegenerate quadrics and the Hilbert scheme $F_{d-k-1}(\cQ_k/S_k^0)$ parameterizes
its maximal isotropic subspaces, hence the Stein factorization is provided by the double cover 
\begin{equation*}
\Spec_{S_k^0}(\cO_{S_k^0} \oplus \cL_k^{d-k} \otimes \det(\cE_k)) \lra S_k^0,
\end{equation*}
where the algebra structure
on $\cO_{S_k^0} \oplus \cL_k^{d-k} \otimes \det(\cE_k)$ is given by $\det(q_k)$.\
Using~\eqref{eq:det-cek}, we see that this double cover coincides with $\tS_k \times_S S_k^0 \to S_k^0$.\
This proves~\eqref{eq:spk-tsk}.

In particular, the double covers $S'_k \to S_k$ and $\tS_k \to S_k$ have  same {rings} of rational functions.\ Since the scheme $\tS_k$ is normal, it is isomorphic to the normalization of $S'_k$.\
Finally, if $F_{d-1}(\cQ/S)$ is normal, the scheme $S'_k$ is also normal, hence $S'_k \cong \tS_k$.}
\end{proof}

\subsection{Application to double covers of symmetroids}
\label{subsection:symmetroids}

The discriminant hypersurface for a linear system of quadratic forms is classically called a {\sf  symmetroid}.\ To be more precise, let $V$ be a vector space of dimension $m$,   let $W$ be a vector space, and let
\begin{equation}
\label{eq:family-of-quadrics}
W \lra \Sym^2V^\vee 
\end{equation} 
be a linear map 
which we think of as a family of quadratic forms in $\P(V)$ parameterized by $\P(W)$.\ We denote by $\cQ \subset \P(W) \times \P(V)$ the corresponding family of quadrics over~$\P(W)$.\ The corresponding symmetroid hypersurface in $\P(W)$ is the discriminant for the map~\mbox{$\cQ \to \P(W)$}.

This fits in our general framework: the map~\eqref{eq:family-of-quadrics} is a family of quadratic forms 
on the projective space $S = \P(W)$ 
in the trivial vector bundle $\cE = V \otimes \cO_{\P(W)}$,
 the  line bundle is  $\cL = \cO_{\P(W)}(-1)$, and the symmetroid is the corank-1 locus $S_1 \subset S = \bP(W)$.\
As an application of our results, we will construct a canonical double cover of this symmetroid.\ These double covers appeared for instance in~\cite[Section~2.3]{HT}.

 {As in the first paragraph of Section~\ref{sec3},} we set  $\cC := \Coker\bigl(V \otimes \cO_{\P(W)}(-1) \to V^\vee \otimes \cO_{\P(W)}\bigr)$ 
and we denote by $S_k\subset \P(W)$   the corank-$k$ locus.

\begin{theorem}
\label{thoerem:symmetroids}
Assume that we are given a linear map as in \eqref{eq:family-of-quadrics}, with $  \dim(V)=2d - 1$  odd.\ Assume moreover that
 $S_1$ is a hypersurface in $\P(W)$, that     $S_1 \setminus S_2$ is nonsingular, and that   $\codim_{S_1}(S_2) \ge 2$.
\begin{itemize}
\item[$1)$] 
There is a double cover 
\begin{equation*}
f \colon \tS_1 \lra S_1,
\end{equation*}
with $\tS_1$ normal, $f$ \'etale over $S_1 \setminus S_2$, and  
\begin{equation*}
f_*\cO_{\tS_1} \cong \cO_{S_1} \oplus \cC(1-d).
\end{equation*}
If the base field is quadratically closed and   $\dim (W) \ge 3$, this double cover  is unique up to isomorphism.
\item[$2)$] 
If moreover $S_2 \setminus S_3$ is nonsingular, $\tS_1$ is nonsingular over $S_1 \setminus S_3$.
\item[$3)$]
If, additionally, $\dim ( {F_{d-1}(\cQ/\P(W)) \times_{\P(W)} S_3}) \le \dim (W) + \frac12 d(d-3) - 3$, 
the cover $f$  provides the Stein factorization for the  map $F_{d-1}(\cQ/\P(W)) \to \P(W)$ defined in~\eqref{rhs}.
\end{itemize}
\end{theorem}

 If $S_3 = \varnothing$, one concludes  in  {statement} 2) that $\tS_1$ is everywhere nonsingular, and  the condition  in  {statement} 3) becomes void.

\begin{proof}
The scheme $S_1$ is normal, because it is a hypersurface in $\P(W)$  which is nonsingular in codimension~1.\
The double cover is then given by Theorem~\ref{theorem:covering-quadratic} (taking $k=1$ and $\cM= \cO_{S_1}(1-d)$): 
we only have to check that $\cR_1 \cong \cC$ or, in view of the definition~\eqref{eq:rk-ck} of $\cR_1$, 
that $\cC$ is a reflexive sheaf on~$S_1$.\
Restricting to $S_1$  the  exact sequence
\begin{equation*}
0 \to V \otimes \cO_{\P(W)}(-1) \to V^\vee \otimes \cO_{\P(W)} \to \cC \to 0
\end{equation*}
of  sheaves on $\P(W)$  gives an exact sequence
\begin{equation*}
0 \to \cC(-m) \to V \otimes \cO_{S_1}(-1) \to V^\vee \otimes \cO_{S_1} \to \cC \to 0.
\end{equation*}
of  sheaves on $S_1$.\ The sheaf   $\cC(-m)$ is therefore  reflexive and  so is $\cC$.\

If $\dim (W) \ge 3$, the projective scheme $S_1$   is integral hence   all regular functions on $S_1$ are constant.\ If  moreover the  base field is quadratically closed, they are squares and 
 the uniqueness of the double cover  follows from Theorem~\ref{theorem:covering-quadratic}.

If $S_2 \setminus S_3$ is nonsingular, 
the family of quadrics $\cQ$ is 2-regular at all points of $S_1 \setminus S_3$ by Lemma~\ref{lemma:regularity}.\
By Proposition~\ref{proposition:covering-smooth}, the scheme $\tS_1$ is nonsingular over $S_1 \setminus S_3$.\ 
This proves 2).

Under the hypotheses of 3), the relative Hilbert scheme $F_{d-1}(\cQ/S)$ is normal by Lemma~\ref{lemma:normality-f}, 
hence, by Proposition~\ref{proposition:isotropic-stein}, the double cover $f$ provides the Stein factorization.
\end{proof}

\section{Double covers of Lagrangian intersection loci}
\label{section:lagrangian}

Let $\cV$ be a vector bundle of rank $2n$ on a scheme $S$ and let $\omega \colon \bw2\cV \to \cL$ be a family 
of symplectic forms on $\cV$ (with values in a line bundle $\cL$).\
A {\sf Lagrangian subbundle} in $\cV$ is a rank-$n$ vector subbundle~$\cA \subset \cV$ 
 such that the composition
\begin{equation*}
\omega_{\cA,\cA} \colon 
 \cA \lhra \cV \xrightarrow[{}^\sim]{\ \omega\ } \cV^\vee \otimes \cL \twoheadrightarrow \cA^\vee \otimes \cL 
\end{equation*}
is zero.\
Consequently, for any Lagrangian subbundle $\cA \subset \cV$, there is an exact sequence
\begin{equation}\label{eq:lagrangian-sequence}
0 \to \cA \to \cV \to \cA^\vee \otimes \cL \to 0,
\end{equation} 
where the   map $\cV \to \cA^\vee \otimes \cL$  is the composition as above.

Let $\cA_1,\cA_2 \subset \cV$ be Lagrangian subbundles.
We define the subscheme
\begin{equation}\label{defsk}
S_k= S_k(\cA_1,\cA_2) \subset S
\end{equation}
as the corank-$k$ degeneracy locus of the morphism
\begin{equation*}
\omega_{\cA_1,\cA_2} \colon 
 \cA_1 \lhra \cV \xrightarrow[{}^\sim]{\ \omega\ } \cV^\vee \otimes \cL \twoheadrightarrow \cA_2^\vee \otimes \cL 
\end{equation*}
and set 
$S_k^0 := S_k \setminus S_{k+1}$.
Set-theoretically, the subscheme $S_k$ parameterizes points of $S$ at which 
the intersection of the fibers of $\cA_1$ and $\cA_2$ has  dimension at least $ k$.\
The subschemes $S_k \subset S$ are called the {\sf Lagrangian intersection loci} of $\cA_1$ and $\cA_2$.

We define the {\sf Lagrangian cointersection sheaf} as the cokernel of the map $\omega_{\cA_1,\cA_2}$.
We will be especially interested in its restrictions to various intersection loci, so we set 
\begin{equation*}
 \cC_k =   \cC_k(\cA_1,\cA_2) := \Coker(\cA_1 \xrightarrow{\ \omega_{\cA_1,\cA_2}\ } \cA_2^\vee \otimes \cL)\vert_{S_k(\cA_1,\cA_2)}
\end{equation*}
The next lemma shows that the subschemes $S_k(\cA_1,\cA_2)$ and the sheaves $\cC_k(\cA_1,\cA_2)$ do not depend on the ordering of the Lagrangian subbundles $\cA_1$ and $\cA_2$.

\begin{lemma}
We have $S_k(\cA_1,\cA_2) = S_k(\cA_2,\cA_1)$  and $\cC_k(\cA_1,\cA_2) \cong \cC_k(\cA_2,\cA_1)$.
\end{lemma}

\begin{proof}
Both sides of the equality (resp.\ of the isomorphism) can be rewritten as degeneracy loci (resp.\ cokernel sheaves) 
of the morphism $\cA_1 \oplus \cA_2 \to \cV$.
\end{proof}

\subsection{Double covers}

We construct  natural double covers of the schemes~$S_k = S_k(\cA_1,\cA_2)$ defined in~\eqref{defsk}.\
As in Section~\ref{subsection:covers-quadratic}, we assume that $S_k$ is normal and that $S_k^0 := S_k \setminus S_{k+1}$ is dense in $S_k$.\ 
We consider the rank-1 reflexive sheaf 
\begin{equation}
\label{eq:rk-ck-lag}
\cR_k \cong (\bw{k}\cC_k)^{\vee\vee}
\end{equation} 
on $S_k$.
 
\begin{theorem}\label{theorem:covering-lagrangian}
Assume that $S_k$ is normal   and that $\codim_{S_k}(S_{k+1}) \ge 2$.\
For each line bundle $\cM$ on~$S_k$ such that 
\begin{equation}
\label{eq:cm-cl-ca}
\big(\cL^{\otimes(-n-k)} \otimes \det(\cA_1) \otimes \det(\cA_2)\big)\big\vert_{S_k} \cong \cM^{\otimes2},
\end{equation} 
there is a double cover $f_\cM \colon \tS_k \to  S_k$, with $\tS_k$ normal, that satisfies  the following two properties:
\begin{itemize}
\item[\textnormal{(a)}] there is an isomorphism
\begin{equation*}
f_{\cM*}\cO_{\tS_k} \cong \cO_{S_k} \oplus (\cM \otimes \cR_k ),
\end{equation*}
\item[\textnormal{(b)}] the morphism $f_\cM$ is \'etale over the dense open subset~$S^0_k=S_k \setminus S_{k+1}$.
\end{itemize}
If all invertible functions on $S_k$ are squares, such a double cover is unique up to isomorphism.
\end{theorem}

\begin{proof}
The argument is analogous to that of Theorem~\ref{theorem:covering-quadratic}: the maps $(\cA_i^\vee \otimes \cL)\vert_{S_k^0 } \to \cC_k\vert_{S_k^0 }$ are epimorphisms of vector bundles, 
so we may consider $(\cC_k^\vee \otimes \cL)\vert_{S_k^0 }$ as a subbundle of both $\cA_1\vert_{S_k^0 }$ and~$\cA_2\vert_{S_k^0 }$.\
Set
\begin{equation*}
\cA_{i,k} := (\cA_i\vert_{S_k^0})/(\cC_k^\vee \otimes \cL)\vert_{S_k^0 }
\qquad\text{and}\qquad 
\cL_k := \cL\vert_{S_k^0}.
\end{equation*}
Note that $\cA_{1,k}$ and $\cA_{2,k}$ are vector bundles of rank $n-k$ on $S_k^0$ with   canonical isomorphisms
\begin{equation*}
\det(\cA_{i,k}) \cong (\cL^{\otimes(-k)}  \otimes \det(\cA_i) \otimes \cR_k)\vert_{S_k^0}.
\end{equation*}
The map $\omega_{\cA_1,\cA_2}$ induces an isomorphism
\begin{equation*}
\omega_k \colon \cA_{1,k} \xrightarrow{\ \sim\ } \cA_{2,k}^\vee \otimes \cL_k
\end{equation*}
 {of sheaves on $S_k^0$}.\
Its determinant gives a  {canonical} isomorphism
\begin{equation*}
\det(\omega_k) \colon \big(\cL^{\otimes(-k)}  \otimes \det(\cA_1) \otimes \cR_k \big)\big\vert_{S_k^0 } \xrightarrow{\ \sim\ } 
\big(\cL^{\otimes(-k)}  \otimes \det(\cA_2) \otimes \cR_k \big)\big\vert^\vee_{S_k^0} \otimes \cL^{\otimes(n-k)} {\vert_{S_k^0}}.
\end{equation*}
Under our assumptions, this provides a self-duality on the line bundle $(\cM \otimes \cR_k)\vert_{S_k^0}$
 which extends uniquely to a self-duality on the rank-1 reflexive sheaf $\cM \otimes \cR_k$ on $S_k$.\
 By Proposition~\ref{proposition:divisor-cover}, it gives a double cover $f_\cM \colon \tS_k \to S_k$, with $\tS_k$ normal,  and $f_\cM$ is \'etale over $S_k^0$.\
The uniqueness also follows from Proposition~\ref{proposition:divisor-cover}.
\end{proof}

{We have the following analogue of Lemma~\ref{lemma:ksk-quadratic}.

\begin{lemma}
\label{lemma:ksk-lagrangian}
Assume that $S_k$ is normal, $\codim_S(S_k) = k(k+1)/2$, and $\codim_{S_k}(S_{k+1}) \ge 2$.\ 
We have {an equality} 
of Weil divisor classes
\begin{equation*}
K_{S_k} = 
\begin{cases}
K_S\vert_{S_k} + \rc_1(\cR_k) - \tfrac{k}2(\rc_1(\cA_1) + \rc_1(\cA_2)) + \tfrac{k(n+k)}2 \rc_1(\cL) & \text{if $k$ is even,}\\
K_S\vert_{S_k} - \tfrac{k+1}2(\rc_1(\cA_1) + \rc_1(\cA_2)) + \tfrac{(k+1)(n+k)}2 \rc_1(\cL) & \text{if $k$ is odd.}
\end{cases}
\end{equation*}
\end{lemma}

\begin{proof}
As in Lemma~\ref{lemma:ksk-quadratic}, the conormal bundle to $S_k^0$ is isomorphic to $\Sym^2(\cC_k)$, 
so we obtain the equality $K_{S_k^0} = K_S\vert_{S_k^0} + (k+1)\rc_1(\cR_k)$.\
The proof of Theorem~\ref{theorem:covering-lagrangian} shows that on~$S_k^0$, we have 
\begin{equation*}
2\rc_1(\cR_k) + \rc_1(\cA_1) + \rc_1(\cA_2) - (n+k)\rc_1(\cL) = 0.
\end{equation*}
Repeating the argument of Lemma~\ref{lemma:ksk-quadratic}, we deduce the required equalities.
\end{proof}

}

\subsection{Isotropic reduction}

Let $\cI \subset \cV$ be an isotropic subbundle of rank~$r$,  that is, a   subbundle such that 
the composition 
$\omega_{\cI,\cI} {\colon \cI \to \cI^\vee \otimes \cL}$
 is zero.\
 Then,
 \begin{equation*}
\bcV := \tcV / \cI,
\quad \textnormal{where}\quad
\tcV := \Ker(\cV \xrightarrow[{}^{\scriptstyle{\sim}}]{\ \omega\ } \cV^\vee \otimes \cL \twoheadrightarrow \cI^\vee \otimes \cL),
\end{equation*}
 is a vector bundle on $S$ of rank $2(n-r)$ and the symplectic form $\omega$ on $\cV$ 
induces a symplectic form 
\begin{equation*}
\bar\omega \colon \bcV \lra \bcV^\vee \otimes \cL.
\end{equation*}
The pair $(\bcV,\bar\omega)$ is called the {\sf isotropic reduction} of $(\cV,\omega)$ with respect to $\cI$.

Let $\cA \subset \cV$ be a Lagrangian subbundle.

\begin{lemma}
Assume that the composition 
\begin{equation}
\label{eq:map-ca-ci}
\cA \lhra \cV \lra \cI^\vee \otimes \cL
\end{equation}
has constant rank.\ Its kernel $\widetilde\cA$ is a subbundle of $\tcV $ 
and  {the image $\bcA$ of $\widetilde\cA$ in $\bcV$} is a Lagrangian subbundle.
\end{lemma}

\begin{proof}
Both $\widetilde\cA$  and   the kernel  {$\cK := \Ker(\cI \to \cA^\vee \otimes \cL)$} 
of the transposed map of~\eqref{eq:map-ca-ci} are locally free.
Consider the commutative diagram
\begin{equation*}
\xymatrix@M=5pt{
& \cA \ar[r] \ar@{_(->}[d] & \cI^\vee \otimes \cL \ar@{=}[d] \\
\cI \ar@{^(->}[r] \ar@{=}[d] & \cV \ar[r] \ar[d] & \cI^\vee \otimes \cL \\
\cI \ar[r] & \cA^\vee \otimes \cL.
}
\end{equation*}
Its rows are complexes and its columns are exact sequences.\ By the snake lemma, it induces a long exact sequence
\begin{equation*}
0 \to \cK \to \widetilde\cA \to \bcV \to \widetilde\cA^\vee \otimes \cL \to \cK^\vee \otimes \cL \to 0.
\end{equation*}
It follows that $\bcA = \widetilde\cA / \cK$ is a Lagrangian subbundle in $\bcV$.
\end{proof}

We call the Lagrangian subbundle $\bcA \subset \bcV$ the {\sf isotropic reduction} of $\cA$ (with respect to $\cI$).

\begin{proposition}\label{proposition:isotropic-reduction}
Let $\cA_1,\cA_2 \subset \cV$ be Lagrangian subbundles and 
let $\cI_1 \subset \cA_1$, $\cI_2 \subset \cA_2$ be subbundles such that 
the morphisms $\cI_1 \oplus \cA_2 \to \cV$ and $\cA_1 \oplus \cI_2 \to \cV$ are embeddings of vector bundles  
\textup(so that the respective quotients are vector bundles\textup)  
and the image $\cI := \Im(\cI_1 \oplus \cI_2 \to \cV)$ is isotropic.

If $\bcA_1$ and $\bcA_2$ are {the isotropic reductions with respect to $\cI$}, we have 
\begin{equation*}
S_k(\bcA_1,\bcA_2) = S_{k}(\cA_1,\cA_2)
\qquad\text{for all $k$}.
\end{equation*}
Moreover, if the scheme $S_k(\cA_1,\cA_2)$ is normal  and $\codim_{S_k(\cA_1,\cA_2)}(S_{k+1}(\cA_1,\cA_2)) \ge 2$, we have,
for any line bundle $\cM$ on $S_k(\cA_1,\cA_2)$ satisfying~\eqref{eq:cm-cl-ca}, an isomorphism
\begin{equation}
\label{eq:cm-cl-bca}
\big(\cL^{\otimes(-(n-r)-k)} \otimes \det(\bcA_1) \otimes \det(\bcA_2)\big)\big\vert_{S_k(\bcA_1,\bcA_2)} \cong \cM^{\otimes2} 
\end{equation}
and the corresponding double covers are isomorphic: there is a commutative diagram
\begin{equation*}
\xymatrix@C=3em{
\tS_k(\bcA_1,\bcA_2) \ar[r]^-\sim \ar[d]_{\bar{f}_\cM} &
\tS_{k}(\cA_1,\cA_2) \ar[d]^{f_\cM}
\\
S_k(\bcA_1,\bcA_2) \ar@{=}[r]  &
S_{k}(\cA_1,\cA_2).
}
\end{equation*}
 \end{proposition}

\begin{proof}
The hypotheses imply that the morphisms $\cI_1 \to \cV \to \cA_2^\vee \otimes \cL$ and $\cI_2 \to \cV \to \cA_1^\vee \otimes \cL$  are embeddings of vector bundles, hence their dual maps are epimorphisms.\
 On the other hand, the maps~$\cI_1 \to \cV \to \cA_1^\vee \otimes \cL$ and~$\cI_2 \to \cV \to \cA_2^\vee \otimes \cL$ are zero, hence so are their duals.\ This means that the image  of the map~$\cA_i \to \cV \to \cI^\vee \otimes \cL$ is $\cI_{3-i}^\vee \otimes \cL$.\  In particular, these maps have constant rank and   the isotropic reductions $\bcA_1$ and $\bcA_2$ are well defined.

Consider the case   $\cI_2 = 0$ and $\cI = \cI_1$.\ The isotropic reductions of $\cA_1$ and $\cA_2$ are then given 
by~$\bcA_1 = \cA_1/\cI$ and $\bcA_2 = \Ker(\cA_2 \to \cI^\vee \otimes \cL)$.\ Therefore,
\begin{equation*}
\det(\bcA_1) \cong \det(\cA_1) \otimes \det(\cI)^\vee 
\quad \textnormal{and}\quad
\det(\bcA_2) \cong \det(\cA_2) \otimes \det(\cI) \otimes \cL^{\otimes(-r)},
\end{equation*}
hence we have~\eqref{eq:cm-cl-bca}.\ Furthermore, there is a commutative diagram
\begin{equation*}
\xymatrix{
0 \ar[r] & \cI \ar[r] \ar@{=}[d] & \cA_1 \ar[r] \ar[d]^{\omega_{\cA_1,\cA_2}} & \bcA_1 \ar[r] \ar[d]^{\omega_{\bcA_1,\bcA_2}} & 0
\\
0 \ar[r] & \cI \ar[r] & \cA_2^\vee \otimes \cL \ar[r] & \bcA_2^\vee \otimes \cL \ar[r] & 0.
}
\end{equation*}
The cointersection sheaves $\cC = \Coker( \omega_{\cA_1,\cA_2})$ and $\overline{\cC} = \Coker ( \omega_{\bcA_1,\bcA_2})$ are therefore isomorphic.\ 
Since the Lagrangian intersection loci are defined via the rank stratification of the cointersection sheaf
(that is, their ideals  are the Fitting ideals of the cointersection sheaf), 
we deduce an equality  {of subschemes} $S_k(\bcA_1,\bcA_2) = S_k(\cA_1,\cA_2)$ for all $k$.

To identify the double covers, we consider, after identifying $\cC_k$ and $\overline{\cC}_k$, the diagram
\begin{equation*}
\xymatrix{
0 \ar[r] & \cI\vert_{S_k^0} \ar[r] \ar@{=}[d] & \cA_{1,k} \ar[r] \ar[d]^{\omega_k} & \bcA_{1,k} \ar[r] \ar[d]^{\bar\omega_k} & 0
\\
0 \ar[r] & \cI\vert_{S_k^0} \ar[r] & \cA_{2,k}^\vee \otimes \cL_k \ar[r] & \bcA_{2,k}^\vee \otimes \cL_k \ar[r] & 0
}
\end{equation*}
It implies  $\det(\omega_k) = \det(\bar\omega_k)$, hence the double covers of~$S_k^0$ before and after the isotropic reduction are the same.\ Since the double covers of $S_k$ are obtained by taking the normal closures, they are the same too.

The case  $\cI_1 = 0$ and $\cI = \cI_2$ can be dealt with in the same way (just switch the roles of~$\cA_1$ and~$\cA_2$).\ Finally,  the general isotropic reduction (when both $\cI_1$ and $\cI_2$ are nonzero) can be done in two steps: first consider the reduction
with respect to $\cI_1$ and then the reduction with respect to $\cI_2$.\ So, applying twice the above argument, we deduce the general claim.
\end{proof}

\subsection{Relation to quadratic covers}

We show  that  Lagrangian intersection loci and their cointersection sheaves can  {locally (and sometimes also globally)} be written as quadratic degeneracy loci and their cokernel sheaves for appropriate families of quadrics.

Let as above   $\cA_1,\cA_2 \subset \cV$ be Lagrangian subbundles and let $\cA_3 \subset \cV$ be another Lagrangian subbundle 
such that  $S_1(\cA_1,\cA_3) = S_1(\cA_2,\cA_3) = \varnothing$, that is, both maps 
\begin{equation}
\label{eq:omega-13-23}
\omega_{\cA_3,\cA_2} \colon \cA_3 \lra \cA_2^\vee \otimes \cL
\qquad\text{and}\qquad
\omega_{\cA_1,\cA_3} \colon \cA_1 \lra \cA_3^\vee \otimes \cL
\end{equation}
are isomorphisms.
We show that \'etale locally, such an $\cA_3$ always exists.

\begin{lemma}
For any closed point $s \in S$, there is an \'etale neighborhood $(U,u) \to (S,s)$ and a Lagrangian subbundle $\cA_3 \subset \cV_U$ 
such that the maps~\eqref{eq:omega-13-23} are isomorphisms.
 \end{lemma}

\begin{proof}
We may assume that the vector bundles $\cA_1$, $\cA_2$, $\cV$, and $\cL$ are trivial.\
Let $V$ be the fiber of~$\cV$ at point $s$ and let $A_1, A_2 \subset V$ be the fibers of $\cA_1$ and~$\cA_2$.\
 For each $i \in \{1,2\}$, the set  of Lagrangian subspaces $A \subset V$ such that $A \cap A_i \ne 0$ 
 is a Schubert hyperplane in the Lagrangian Grassmannian~$\LGr(V)$, 
hence one can choose a Lagrangian subspace $A \subset V$ such that 
\begin{equation*}
A \cap A_1 = A \cap A_2 = 0.
\end{equation*}
Since  {the projection $\LGr_S(\cV) \to S$ of} the relative Lagrangian Grassmannian  is  smooth,  it has, locally in the \'etale topology, 
a section passing through the point~$[A]$ in the fiber over~$s$.\ We define $\cA_3$ to be the corresponding Lagrangian subbundle.\ The maps~\eqref{eq:omega-13-23} are isomorphisms at  $s$ by definition of $\cA_3$.\ Shrinking~$S$ if necessary, we may assume that they are isomorphisms on $S$.
 \end{proof}

Assume that $\cA_3$ is chosen so that the maps in~\eqref{eq:omega-13-23} are isomorphisms.\ The composition
\begin{equation}
\label{eq:q-omega}
\cA_3^\vee \otimes \cL \xrightarrow[{}^\sim]{\  {\omega_{\cA_1,\cA_3}^{-1}}\ }
\cA_1 \xrightarrow{\ \omega_{\cA_1,\cA_2}\ }
\cA_2^\vee \otimes \cL  \xrightarrow[{}^\sim]{\  {\omega_{\cA_3,\cA_2}^{-1}}\ }
\cA_3 
\end{equation}
gives a family of bilinear forms 
\begin{equation*}
q \colon \cL \lra \cE^\vee \otimes \cE^\vee
\end{equation*}
on the rank-$n$ vector bundle 
\begin{equation*}
\cE := \cA_3^\vee.
\end{equation*}

\begin{proposition}
\label{proposition:lag-quad}
The family of bilinear forms $q$ is symmetric.\
Moreover, the Lagrangian intersection loci and cointersection sheaves 
coincide with the corresponding quadratic degeneracy loci and cokernel sheaves:
\begin{equation*}
S_k(\cA_1,\cA_2) = S_k(q) 
\qquad  {and}\qquad
\cC_k(\cA_1,\cA_2) \cong \cC_k(q).
\end{equation*}
Denoting the   scheme $S_k(q)$ simply by $S_k$, we have,
for any line bundle $\cM$ on $S_k$ satisfying the isomorphism~\eqref{eq:cm-cl-ca},
\begin{equation}
\label{eq:cm-cl-q}
\cL^{\otimes(n-k)}\vert_{S_k} \cong (\cM \otimes \det(\cE)^\vee)^{\otimes2}.
\end{equation}
Finally, if $S_k$ is normal  and $\codim_{S_k}(S_{k+1}) \ge 2$,
the double covers  {in Theorems~\textup{\ref{theorem:covering-quadratic}} and~\textup{\ref{theorem:covering-lagrangian}}
respectively associated with the line bundles $\cM \otimes \det(\cE)^\vee$ and $\cM$} are isomorphic: there is a commutative diagram
\begin{equation*}
\xymatrix@C=3em{
\tS_k(q) \ar[r]^-\sim \ar[d]_{f_{\cM \otimes \det(\cE)^\vee}} &
\tS_{k}(\cA_1,\cA_2) \ar[d]^{f_\cM}
\\
S_k(q) \ar@{=}[r]  &
S_{k}(\cA_1,\cA_2).
}
\end{equation*}
 \end{proposition}

\begin{proof}
The symmetry of $q$ is checked by a standard computation.\ 
Since the first and  last maps in~\eqref{eq:q-omega} are isomorphisms, 
the Lagrangian cointersection sheaf of $(\cA_1,\cA_2)$ and  the cokernel sheaf of~$q$ are isomorphic.\
The loci $S_k(q)$ and~$S_k(\cA_1,\cA_2)$  being defined via the rank stratification of these sheaves, 
the equality of subschemes~\mbox{$S_k(\cA_1,\cA_2) = S_k(q)$} and the isomorphism~$\cC_k(\cA_1,\cA_2) \cong \cC_k(q)$ follow.\
Furthermore, the isomorphisms~\eqref{eq:omega-13-23} give
\begin{equation}
\label{eq:det-isomorphisms}
\det(\cA_1) \cong \det(\cA_2) \cong \det(\cA_3^\vee) \otimes \cL^{\otimes n} = \det(\cE) \otimes \cL^{\otimes n},
\end{equation}
hence any line bundle $\cM$ satisfying~\eqref{eq:cm-cl-ca} also satisfies~\eqref{eq:cm-cl-q}.\ In particular, Theorems~\ref{theorem:covering-lagrangian} and~\ref{theorem:covering-quadratic} provide double covers $f_\cM$ and~$f_{\cM \otimes \det(\cE)^\vee}$ over $S_k$.

To identify these double covers, we   consider the diagram 
 \begin{equation*}
\xymatrix@C=5em{
\cA_{3,k}^\vee \otimes \cL_k \ar[r]^-{\omega_{\cA_1,\cA_3}^{-1}} \ar@{=}[d] & 
\cA_{1,k} \ar[r]^{\omega_k} &
\cA_{2,k}^\vee \otimes \cL_k \ar[r]^-{\omega_{\cA_3,\cA_2}^{-1}} &
\cA_{3,k} \ar@{=}[d] 
\\
\cE_k \otimes \cL_k \ar[rrr]^{q_k} &&& 
\cE_k^\vee 
}
\end{equation*}
of sheaves   on $S_k^0$  obtained by taking the quotients of the  two leftmost  terms in~\eqref{eq:q-omega} 
by $ \cC_k^\vee \vert_{S_k^0} \otimes \cL_k$ and by considering the kernels of the maps 
from the  two rightmost terms in~\eqref{eq:q-omega} into $ \cC_k \vert_{S_k^0}$.\ 
This implies
\begin{equation*}
\det(q_k) = \det(\omega_{\cA_1,\cA_3})^{{-1}} \otimes \det(\omega_{\cA_3,\cA_2})^{{-1}} \otimes \det(\omega_k).
\end{equation*}
The first two factors induce isomorphisms in~\eqref{eq:det-isomorphisms}, so the above equality means
that the self-duality isomorphisms of the reflexive sheaves $\cM \otimes \cR_k$ 
used in Theorems~\ref{theorem:covering-quadratic} and~\ref{theorem:covering-lagrangian} coincide.\
Therefore, the double covers coincide as well.
\end{proof}

From this proposition, we deduce a useful  nonsingularity criterion for Lagrangian double covers
(the branch and ramification loci were defined in Definition~\ref{definition:branch-ramification}).

\begin{corollary}\label{corollary:lagrangian-smooth}
Assume that  $S$ is nonsingular, that $S_i \setminus S_{i+1}$ is nonsingular of codimension \mbox{$i(i+1)/2$} in~$S$ for each $i \in\{k, k+1\}$, 
and that $S_{k+2}=\varnothing$.\ 
For each choice of a line bundle $\cM$  on $S_k$ satisfying~\eqref{eq:cm-cl-ca}, giving rise to a double cover  $f_\cM \colon \tS_k \to S_k$, 
we have:
\begin{itemize}
\item the scheme $\tS_k$  is nonsingular,
\item the branch locus of $f_\cM $ is equal to $S_{k+1}$,
\item the preimage of the branch locus is the first order infinitesimal neighborhood of the ramification locus.
\end{itemize}
\end{corollary}

\begin{proof}
The statement is \'etale  local, so we may use the local quadratic presentation of $\tS_k$  provided by Proposition~\ref{proposition:lag-quad}.\ By Lemma~\ref{lemma:regularity}, the corresponding family of quadrics is $(k+1)$-regular on  {$S_k$},
so Proposition~\ref{proposition:covering-smooth} gives all we need.
\end{proof}

\section{Application to EPW varieties}

We apply the results of the previous sections to several Lagrangian intersection loci 
related to the choice of a Lagrangian subspace in a certain 20-dimensional symplectic vector space.\ 
Some of these loci appeared in the article~{\cite{epw}} of Eisenbud, Popescu, and Walter, as examples of codimension~3 subvarieties 
that are not quadratic degeneracy loci.\ 
For this reason, they are called {\sf Eisenbud--Popescu--Walter} loci, or EPW loci for short.\ 
We will use various results of O'Grady from \cite{og1,og3,og4,og5,og6,og7}, so we work over the field of complex numbers.

Let $V_6$ be a vector space of dimension $6$.\ We endow the 20-dimensional space~$\bw3V_6$ with
  the symplectic form    given by  wedge product  
(it takes values in $\det(V_6)$ and we trivialize this space by choosing a volume form on $V_6$).\

\subsection{EPW stratification of $\P(V_6)$}

Let $A \subset \bw3V_6$ be a Lagrangian subspace.\  We say that {\sf $A$ has no decomposable vectors}, if
\begin{equation*}
\P(A) \cap \Gr(3,V_6) = \varnothing,
\end{equation*}
where the intersection takes place inside $\P(\bw3V_6)$.\ We consider two Lagrangian subbundles of the trivial symplectic vector bundle $\cV = \bw3V_6 \otimes \cO_{\P(V_6)}$ on $\P(V_6)$.\ The first is the trivial bundle $\cA_1:=A \otimes \cO_{\P(V_6)}$.\  The second, $\cA_2 := \bw2T_{\P(V_6)}(-3)$, comes from the truncation 
\begin{equation*}
0 \to \bw2T_{\P(V_6)}(-3) \to \bw3V_6 \otimes \cO_{\P(V_6)} \to \bw3T_{\P(V_6)}(-3) \to 0 
\end{equation*}
of the Koszul complex (or, equivalently, from the exterior cube of the Euler sequence).\  The fiber of~$\bw2T_{\P(V_6)}(-3)$ at a point $v \in \P(V_6)$  is  $v \wedge \bw2(V_6/\C v) \subset \bw3V_6$, 
hence it is indeed a Lagrangian subbundle of $\cV$.\ One can consider the Lagrangian intersection loci for these two Lagrangian subbundles,   their cokernel sheaves, and the induced double covers.

 The traditional notation for the Lagrangian intersection loci in this case is
\begin{equation}\label{eq:def-yk}
\sY_A^{\ge k}  := S_k\bigl({\cA_1,\cA_2}\bigr) \subset \P(V_6)
\qquad \textnormal{and}\qquad
\sY_A^k  :=\sY_A^{\ge k}  \setminus \sY_A^{\ge k+1} = S_k^0\bigl({\cA_1,\cA_2}\bigr).
\end{equation}
The results of O'Grady that we need  can be summarized as follows (\cite[Theorem~B.2]{DK1}); 
the various singular loci are endowed with  {their reduced scheme structures}.

\begin{theorem}\label{theorem:ogrady}
If the Lagrangian $A$ has no decomposable vectors, the following properties hold:
\begin{itemize}
\item[\rm (a)] $\sY_A^{\ge 1} $ is an integral normal sextic hypersurface in $\P(V_6)$;
\item[\rm (b)] $\sY_A^{\ge 2} $ is the singular locus of $ \sY_A^{\ge 1}  $; 
it is an integral normal  Cohen--Macaulay surface of degree~$40$;
\item[\rm (c)] $\sY_A^{\ge 3} $ is the singular locus of $  \sY_A^{\ge 2} $;
 {it} is  finite and smooth,  and is empty for  $A$ general;
\item[\rm (d)] $\sY_A^{\ge 4} $  is empty.
\end{itemize}
\end{theorem}

In this situation, the  line bundles $\cL$ and $\det (\cA_1) = \det(A \otimes \cO_{\P(V_6)})$  are both trivial, while 
\begin{equation*}
\det(\cA_2) = \det(\bw2T_{\P(V_6)}(-3)) \cong \cO_{\P(V_6)}(-6).
\end{equation*}
The line bundle $\cL^{\otimes(-10-k)} \otimes \det(\cA_1) \otimes \det(\cA_2) \cong \cO_{\P(V_6)}(-6)$ of Theorem~\ref{theorem:covering-lagrangian} 
therefore has a unique square root, $\cO_{\P(V_6)}(-3)$.\
We always take for $\cM$  the restriction of $\cO_{\P(V_6)}(-3)$.\ Theorem~\ref{theorem:covering-lagrangian} 
 gives the following result (the sheaves $\cR_k$ on $\sY_A^{\ge k}$ were defined by~\eqref{eq:rk-ck-lag}).

\begin{theorem}\label{theorem:y-covers}
If the Lagrangian $A$ has no decomposable vectors, the following properties hold.
\begin{itemize}
\item[$0)$] 
There is a unique double cover $f_0 \colon \widetilde\sY_A^{\ge 0} \to \P(V_6)$ with branch locus $\sY_A^{\ge 1}$ such that 
\begin{equation*}
f_{0*}\cO_{\widetilde\sY_A^{\ge 0}} \cong \cO_{\P(V_6)} \oplus \cO_{\P(V_6)}(-3).
\end{equation*}
The scheme $\widetilde\sY_A^{\ge 0}$ is integral and normal, and it is smooth away from $f_0^{-1}(\sY_A^{\ge 2})$.
  \item[$1)$] 
There is a unique double cover $f_1 \colon \widetilde\sY_A^{\ge 1} \to \sY_A^{\ge 1}$ with branch locus $\sY_A^{\ge 2}$ such that
\begin{equation*}
f_{1*}\cO_{\widetilde\sY_A^{\ge 1}} \cong \cO_{\sY_A^{\ge 1}} \oplus \cR_1(-3).
\end{equation*}
The scheme $\widetilde\sY_A^{\ge 1}$ is  integral and normal, and it is smooth away from $f_1^{-1}(\sY_A^{ 3})$.
\item[$2)$] 
There is a unique double cover $f_2 \colon \widetilde\sY_A^{\ge 2} \to \sY_A^{\ge 2}$ with branch locus $\sY_A^{ 3}$ such that
\begin{equation*}
f_{2*}\cO_{\widetilde\sY_A^{\ge 2}} \cong \cO_{\sY_A^{\ge 2}} \oplus \cR_2(-3).
\end{equation*}
The scheme $\widetilde\sY_A^{\ge 2}$ is integral and normal, it is smooth away from~$f_2^{-1}(\sY_A^{ 3})$
 and has ordinary double points along~$f_2^{-1}(\sY_A^{ 3})$.\ Moreover, $\cR_2 \cong \omega_{\sY_A^{\ge 2}}$.
\end{itemize}
 \end{theorem}

\begin{proof}
For part~0), we let $\widetilde\sY_A^{\ge 0} \to \P(V_6)$ be the double cover branched along the sextic $\sY_A^{\ge 1}$.\ The fact that $\widetilde\sY_A^{\ge 0}$ is smooth away from $f_0^{-1}(\sY_A^{\ge 2})$ follows from Theorem~\ref{theorem:ogrady}(b).\
Integrality  and normality of~$\widetilde\sY_A^{\ge 0}$ are standard.

For part~1), we apply Theorem~\ref{theorem:covering-lagrangian}: 
$\sY_A^{\ge 1}$ is normal and integral by Theorem~\ref{theorem:ogrady}(a), 
the codimension of the next stratum is~2, and we take for the line bundle $\cM$  the restriction of $\cO_{\P(V_6)}(-3)$.\ By Corollary~\ref{corollary:lagrangian-smooth}, ${\widetilde\sY_A^{\ge 1}}$  is smooth away from~$f_1^{-1}(\sY_A^3)$ 
and the branch locus of $f_1$ is equal to $\sY_A^{\ge 2}$.\ Since the branch locus is nonempty, $\cR_1$ is not locally free by Proposition~\ref{proposition:divisor-cover}(b),
hence $\widetilde\sY_A^{\ge 1}$ is integral  by Lemma~\ref{lemma:int}.\ For uniqueness, note that $\sY_A^{\ge 1}$ is proper and connected, hence any regular function on it is constant
 {and, since the base field is $\C$, every constant is a square.}

For part~2), we  apply again Theorem~\ref{theorem:covering-lagrangian} 
(whose hypotheses   are verified as in the previous case) 
to construct the double cover $f_2\colon \widetilde\sY_A^{\ge 2} \to \sY_A^{\ge 2}$, where $\widetilde\sY_A^{\ge 2} $ is normal.\ 
Since $\sY_A^2$   is smooth by Theorem~\ref{theorem:ogrady}(c) and $f_2$ is \'etale over this open subset, 
$\widetilde\sY_A^{\ge 2}$ is smooth away from $f_2^{-1}(\sY_A^3)$.\  The isomorphism between $\cR_2$ and the dualizing sheaf of $\sY_A^{\ge 2}$ follows from Lemma~\ref{lemma:ksk-lagrangian}.\
For the description of the branch locus, {the description of the singularities} when $\sY_A^3 \ne \varnothing$ 
(so that $\sY_A^3$ does not have the expected codimension), and the integrality of $\widetilde\sY_A^{\ge 2}$ when $\sY_A^3 = \varnothing$, we use a relative version of the same construction.

Let $U \subset \LGr(\bw3V)$ be an open subset containing {the point} $[A]$  such that for all $[A'] \in U$, 
the subspace~\mbox{$A' \subset \bw3V_6$} has no decomposable vectors.\ We may also assume that the tautological subbundle~\mbox{$\cA \subset \bw3V_6 \otimes \cO_U$} has trivial determinant 
 {(this assumption is used below to construct a relative double cover, and since it is not satisfied
on the open subset $\LGr_0(\bw3V_6) \subset \LGr(\bw3V_6)$ parameterizing Lagrangian subspaces with no decomposable vectors,
we have to restrict to a smaller open subset $U \subset \LGr_0(\bw3V_6)$)}
and that the open subset~$U_0 \subset U$ of points $[A'] \in U$ such that $\sY_{A'}^3 = \varnothing$ is distinct from $ U$.

The intersection loci for the Lagrangian subbundles  {of $\bw3V_6 \otimes \cO_{U \times \P(V_6)}$}
\begin{equation*}
\cA_1 = \cA \boxtimes \cO_{\P(V_6)}
\qquad\text{and}\qquad
\cA_2 = \cO_U \boxtimes \bw2T_{\P(V_6)}(-3)
\end{equation*}
are given by the total spaces of the EPW strata:
\begin{equation*}
\varnothing = \cY_U^{\ge 4} \ne \cY_U^{\ge 3} \subset \cY_U^{\ge 2} \subset \cY_U^{\ge 1} \subset U \times \P(V_6),
\end{equation*}
where the fiber of $\cY_U^{\ge k}$ over a point $[A] \in U$ is equal to $\sY_A^{\ge k}$.\ Since the scheme $\cY_U^{\ge 2}$ has expected codimension, it is Cohen--Macaulay.\ It is smooth outside the next stratum $\cY_U^{\ge 3}$, which has codimension~3.\ It is therefore normal and, since its fibers over $U$ are integral (Theorem~\ref{theorem:ogrady}(b)), it is also integral.

Since $\det(\cA)$ is trivial on $U$, the line bundle $\cM = \cO_U \boxtimes \cO_{\P(V_6)}(-3)$ satisfies~\eqref{eq:cm-cl-ca} for $k = 2$.\ 
{The next stratum $\cY_U^{\ge 3}$ is smooth of codimension~3 in~$\cY_U^{\ge 2}$ by~\cite[Corollary~2.4]{og4}}.\
{Therefore,} by Theorem~\ref{theorem:covering-lagrangian}, 
there exists a double cover 
\begin{equation*}
\varphi_2 \colon \widetilde\cY_U^{\ge 2} \lra \cY_U^{\ge 2} ,
\end{equation*}
\'etale over~$\cY_U^2$.\ 
Since $ \cY_U^{\ge 4}$ is empty, by Corollary~\ref{corollary:lagrangian-smooth}, $\widetilde\cY_U^{\ge 2}$ is smooth  
and the branch locus of $\varphi_2$ is equal to $\cY_U^3$.\ 
Since this branch locus is nonempty (because $U_0 \ne U$), the scheme $\widetilde\cY_U^{\ge 2}$ is integral
by  {Proposition~\ref{proposition:divisor-cover} and Lemma~\ref{lemma:int}}.

Since the formation of the Lagrangian cointersection sheaf and of the branch locus is compatible with base changes,
and so is the operation of taking the top wedge power,  the restriction of the double cover $\varphi_2$
over a point $[A] \in U$ coincides with the double cover $f_2$ discussed earlier and  the branch locus of $f_2$ is equal to $\sY_A^3$.

Assume that $\widetilde\sY_A^{\ge 2}$ is not integral.\ 
By Lemma~\ref{lemma:int} and  {Proposition~\ref{proposition:divisor-cover}}, 
we have $\sY_A^3 = \varnothing$, that is \mbox{$[A] \in U_0$},
and the corresponding reflexive sheaf $\cR_A$ is trivial, that is, $\cR_A \cong \cO_{\sY_A^{\ge 2}}$.\ 
Since for each~\mbox{$[A'] \in U_0$}, the sheaf $\cR_{A'}$ is a 2-torsion line bundle on the smooth projective surface $\sY_{A'}^{\ge 2}$, 
it follows that $\cR_{A'} \cong \cO_{\sY_{A'}^{\ge 2}}$.\ 
Therefore, the   reflexive sheaf $\cR$ on $\cY_U^{\ge 2}$,  when restricted to $\cY_{U_0}^{\ge 2}$, is the pullback of a line bundle on~$U_0$.\ 
Since the divisor $\cY_U^2 \setminus \cY^2_{U_0}$ on $\cY^2_U$ is the pullback of the divisor~$U \setminus U_0$ on $U$,
it follows  that the line bundle $\cR\vert_{\cY_U^2}$   is isomorphic to the pullback of a line bundle on~$U$.\
Therefore, there is a line bundle $\cL$ on~$\cY_U^{\ge 2}$ such that the rank-1 reflexive sheaf $\cR$ is isomorphic to $\cL$
on the complement of the codimension~3 subset $\cY_U^{\ge3}$, and hence $\cR \cong \cL$ on the entire $\cY_U^{\ge 2}$.\ 
Such an isomorphism contradicts   Lemma~\ref{lemma:cover-divisor}, 
since the double cover $\varphi_2$ has nontrivial branch locus.\ 
This proves that $\widetilde\sY_A^{\ge 2}$ is   integral.

 {Finally, let us describe singularities of $\widetilde\sY_A^{\ge 2}$.\
 Let $[v] \in \sY^3_A$ and set $K := A \cap (v \wedge \bw2V_6)$, so that~$\dim(K) = 3$.\
By~\cite[Corollary~2.4, Proposition~2.5]{og4}, a transversal slice to $\cY^3_U$ in \mbox{$\LGr(\bw3V_6) \times \P(V_6)$} at $([A],[v])$ 
can be identified with the affine space $\SK = \Sym^2(K^\vee)$ and a transversal slice to $\sY^3_A$ in~$\P(V_6)$ at $[v]$
with the hyperplane in $\SK$ corresponding to a nondegenerate quadratic form.\
Furthermore, a transversal slice to $\cY^{\ge 2}_U$ can be identified with the subscheme $\SK_2 \subset \SK$ 
of quadratic forms of corank~$\ge2$ (that is, of rank~$\le 1$) and, by Lemma~\ref{lemma:veronese-cover},
a transversal slice to the double cover $\widetilde\cY^{\ge 2}_U \to \cY^{\ge 2}_U$ 
can be identified with the quotient by the $(\pm 1)$-action map $K^\vee \to \SK_2$.\
Therefore, a transversal slice to $\widetilde\sY^{\ge 2}_A$ at $[v]$ can be identified 
with the affine quadratic cone over a nondegenerate quadric.\
Thus, the scheme~$\widetilde\sY^{\ge 2}_A$ has an ordinary double point at $[v]$.}
\end{proof}

The double cover in part 1) of Theorem~\ref{theorem:y-covers} coincides with the EPW double sextic 
defined by O'Grady---its definition is just the same.\
The double cover in part 2) is  new, although $\widetilde\sY_A^{\ge 2}$ can be interpreted   as the minimal model of the surface of conics on a   Gushel--Mukai threefold (studied in \cite[Proposition~0.1]{log} and \cite[Section~6]{dim}).\  

\begin{remark}\upshape
The    {results of Theorem~\ref{theorem:y-covers}} hold for the dual EPW stratification of $\P(V_6^\vee)$ associated with  the dual Lagrangian subspace $A^\perp = \Ker(\bw3V_6^\vee \to A^\vee)\subset \bw3V_6^\vee$.\ Since $A^\perp$ has no decomposable vectors if and only if $A$ has the same property,  we obtain double covers
\begin{equation*}
\widetilde\sY_\Ap^{\ge k} \lra \sY_\Ap^{\ge k}
\end{equation*}
with  analogous properties.
\end{remark}

\subsection{The first quadratic fibration  {of Gushel--Mukai varieties}}
\label{subsection:fibration-1}
 Let $X$ be a smooth Gushel--Mukai variety  {of dimension~$n\ge3$}, as defined in~\cite{DK1},
and let $(V_6,V_5,A)$ be the Lagrangian data~(\cite[Definition~3.4]{DK1}) associated with~$X$ by~\cite[Theorem~3.10]{DK1}: $V_6$ is a 6-dimensional vector space, $V_5 \subset V_6$ is a hyperplane, and~\mbox{$A \subset \bw3V_6$} is a Lagrangian subspace  {with no decomposable vectors}.\ Set, for any~$k$,
\begin{equation*}
\sY_{A,V_5}^{\ge k} := \sY_A^{\ge k} \cap \P(V_5) 
\qquad\textnormal{and}\qquad
\sY_{A,V_5}^k := \sY_A^k \cap \P(V_5) 
\end{equation*}
 and, for $0 \le k \le 2$,
\begin{equation*}
\widetilde\sY_{A,V_5}^{\ge k} := \widetilde\sY_A^{\ge k} \times_{\P(V_6)} \P(V_5) 
\end{equation*}
(see Theorem~\ref{theorem:y-covers}).\ 

 In~\cite[Section~4.2]{DK1}, we defined a morphism
\begin{equation*}
\rho_1 \colon \cQ_1(X) = \P_X(\cU_X) \lra \P(V_5),
\end{equation*}
(here $\cU_X$ is the rank-2 Gushel bundle of $X$), called {\sf the first quadratic fibration} of $X$,
and a subscheme~$\Sigma_1(X) \subset \P(V_5)$ on the complement of which $\rho_1$ is flat.\ {Ordinary and special Gushel--Mukai varieties are defined in \cite[Section~2.5]{DK1}.}

\begin{lemma}
\label{lemma:q1-lag}
Assume  $n \ge 3$ if $X$ is ordinary and $n \ge 4$ if $X$ is special.\ The degeneracy loci of the first quadratic fibration $\rho_1 \colon \cQ_1(X) \to \P(V_5)$
coincide with the schemes $\sY_{A,V_5}^{\ge k}$ away from $\Sigma_1(X)$,
and their double covers associated with $\rho_1$ coincide with the double covers $f_k \colon \widetilde\sY_{A,V_5}^{\ge k} \to \sY_{A,V_5}^{\ge k}$  {over the complement of $\Sigma_1(X)$}.
 \end{lemma}

\begin{proof}
Assume first that $X$ is ordinary.\  Then $A$ has no decomposable vectors~(\cite[Theorem~3.16]{DK1}) and~$\dim(A \cap \bw3V_5) = 5 - n \le 2$ (\cite[Proposition~3.13]{DK1}).\ Consider the restriction to $\P(V_5)$ of the symplectic vector bundle $\bw3V_6 \otimes \cO_{\P(V_6)}$, 
the Lagrangian subbundles
\begin{equation*}
\cA_1 = A \otimes \cO_{\P(V_5)},
\quad 
\cA_2 = \bw2T_{\P(V_6)}(-3)\vert_{\P(V_5)},
\quad\text{and}\quad
\cA_3 = \bw3V_5 \otimes \cO_{\P(V_5)},
\end{equation*}
and the isotropic subbundles
\begin{equation*}
\cI_1 := \cA_1 \cap \cA_3 = (A \cap \bw3V_5) \otimes \cO_{\P(V_5)} \subset \cA_1
\quad\text{and}\quad
\cI_2 := \cA_2 \cap \cA_3 = \bw2T_{\P(V_5)}(-3) \subset \cA_2
\end{equation*}
of respective ranks $5 - n$ and $6$.\ The natural morphisms
\begin{equation*}
\cA_1 \oplus \cI_2 \lra \bw3V_6 \otimes \cO_{\P(V_5)}
\qquad\text{and}\qquad
\cI_1 \oplus \cA_2 \lra \bw3V_6 \otimes \cO_{\P(V_5)}
\end{equation*}
are embeddings of vector bundles   away from the subscheme $\Sigma_1(X) \subset \P(V_5)$.\ Moreover,  {$\cI = \cI_1 \oplus \cI_2 $,   a subbundle of $ \cA_3$,   is isotropic of rank~$11 - n$}.\ Therefore, over~$\P(V_5) \setminus \Sigma_1(X)$, the conditions of Proposition~\ref{proposition:isotropic-reduction} are satisfied 
and we can perform isotropic reduction with respect to~$\cI$.\ We deduce an isomorphism of double covers
\begin{equation*}
\widetilde\sY_{A, {V_5}}^{\ge k} = \tS_k(\cA_1,\cA_2) \cong \tS_k(\bcA_1,\bcA_2)
\end{equation*}
over  {$S_k = \sY_{A, V_5 }^{\ge k}\setminus \Sigma_1(X)$}.\ Furthermore, as $\cI \subset \cA_3$, the isotropic reduction $\bcA_3 = \cA_3/\cI$ is well defined,  {has rank $n - 1$}, 
and on~$\P(V_5) \setminus \Sigma_1(X)$, both maps
\begin{equation*}
\bar\omega_{\bcA_3,\bcA_2} \colon \bcA_3 \lra \bcA_2^\vee
\qquad\text{and}\qquad
\bar\omega_{\bcA_1,\bcA_3} \colon \bcA_1 \lra \bcA_3^\vee
\end{equation*}
(defined as in~\eqref{eq:omega-13-23}) are isomorphisms.\ The quadratic fibration 
 associated with the Lagrangian subbundles $\bcA_1$, $\bcA_2$, $\bcA_3$ by Proposition~\ref{proposition:lag-quad}  
coincides with the restriction to $\P(V_5)\setminus \Sigma_1(X)$ of the first quadratic fibration $\rho_1$ of $X$ 
(see~\cite[proof of Proposition~4.5]{DK1}).\ By Proposition~\ref{proposition:lag-quad}, the degeneracy loci of $\rho_1$ coincide with $\sY_{A,V_5}^{\ge k}$,
and the double covers of $\sY_{A,V_5}^{\ge k}$ associated with $\rho_1$ coincide with $\widetilde\sY_{A,V_5}^{\ge k}$.

 Assume now that $X$ is special and let $X_0$ be the associated ordinary Gushel--Mukai variety 
(it has the same Lagrangian data as $X$ and $\Sigma_1(X) = \Sigma_1(X_0)$).\ 
The first quadratic fibrations $\cQ_1(X)$ and $\cQ_1(X_0)$ are related as follows:
if $\cE$ is the rank-$(n - 1)$ vector bundle on $\P(V_5) \setminus \Sigma_1(X)$ 
such that~$\cQ_1(X_0)$ is defined inside $\P(\cE)$ by the quadratic form $q \colon \cO \to \Sym^2(\cE^\vee)$,
the quadratic fibration~$\cQ_1(X)$ is defined inside $\P(\cE \oplus \cO)$ by the quadratic form 
\begin{equation*}
\bar{q} = q \oplus \id \colon \cO \lra \Sym^2(\cE^\vee) \oplus \cO \subset \Sym^2((\cE \oplus \cO)^\vee).
\end{equation*}
 Therefore, the degeneracy loci of $\cQ_1(X)$ coincide with those of $\cQ_1(X_0)$, that is, with~$\sY_{A,V_5}^{\ge k}$.\ Furthermore, the cokernel sheaves of $q$ and $\bar{q}$ are isomorphic, hence the double covers agree.\ Thus, the double covers of $\sY_{A,V_5}^{\ge k}$ associated with $\cQ_1(X)$ coincide with $\widetilde\sY_{A,V_5}^{\ge k}$.
\end{proof}

The Hilbert schemes $F_{d-1}(\cQ_1(X)/\P(V_5))$ were identified in~\cite[Proposition~4.1]{DK2} 
with some irreducible components of the Hilbert schemes $F_{d-1}(X)$ of $(d-1)$-dimensional linear spaces on $X$.\
 The connected fibers of its Stein factorization over $\P(V_5)$ were described in~\cite[Theorems~4.2, 4.3, and~4.7]{DK2}.\ The next corollary identifies the finite morphism in the Stein factorization in cases when it is not trivial,
in particular answering to~\cite[Remark~4.4]{DK2}.

\begin{corollary}
\label{corollary:q1-lag}
In the situation  {of Lemma~\textup{\ref{lemma:q1-lag}}}, 
let moreover $d$ be an integer such that 
\begin{equation*}
 {(n,d) \in \{ (4,2), (5,2), (5,3), (6,3) \},}
\end{equation*}
and set 
\begin{equation*}
k := 2d + 1 - n.
\end{equation*}
When $k = 2$, assume that $\sY_{A,V_5}^2$ is smooth.\ The double cover $f_k \colon \widetilde\sY_{A,V_5}^{\ge k} \to \sY_{A,V_5}^{\ge k}$ then provides the Stein factorization 
for the map $F_{d-1}(\cQ_1(X)/\P(V_5)) \to \P(V_5)$ over the open subset $\P(V_5) \setminus (\Sigma_1(X) \cup \sY_{A,V_5}^{3})$.
\end{corollary}

\begin{proof}
 {The rank of the vector bundle in the projectivization of which $\cQ_1(X)$ is contained 
is  $m = n - 1$.\ We have $m - k = 2(n - 1 - d)$, which is positive in all cases.\ Therefore, by Proposition~\ref{proposition:isotropic-stein} and Lemma~\ref{lemma:q1-lag}},
 the double cover coming from the Stein factorization agrees 
with the double cover $\widetilde\sY_{A,V_5}^{\ge k} \to \sY_{A,V_5}^{\ge k}$ coming 
from Lagrangian intersection up to normalization.\ It is therefore enough to check that 
both~$\widetilde\sY_{A,V_5}^{\ge k}$ and~$F_{d-1}(\cQ_1(X)/\P(V_5))$ are normal 
over the complement of~\mbox{$\Sigma_1(X) \cup \sY_{A,V_5}^{\ge 3}$}.\

 {For each $k\in\{0,1,2\}$}, the  scheme $\widetilde\sY_{A,V_5}^{\ge k}$ is a Cartier divisor 
in a normal variety $\widetilde\sY_A^{\ge k}$, hence satisfies  condition $\mathbf{S}_2$.\ On the other hand, by~\cite[Lemma~2.5]{DK2}, its singular set has codimension at least~2 away from $\Sigma_1(X)$, 
hence $\widetilde\sY_{A,V_5}^{\ge k} \setminus \Sigma_1(X)$ is normal.

We now  check that $F_{d-1}(\cQ_1(X)/\P(V_5))$ is normal.\ If $k = 2$, $\cQ_1$ is 2-regular by Lemma~\ref{lemma:regularity}, 
because~$\sY_{A,V_5}^{\ge 1} \setminus (\Sigma_1(X) \cup \sY_{A,V_5}^{\ge 2})$ is smooth by~\cite[Lemma~2.5]{DK2}
and $\sY^2_{A,V_5}$ is smooth by assumption.\ Therefore, over $\P(V_5) \setminus (\Sigma_1(X) \cup \sY_{A,V_5}^3)$, 
the scheme $F_{d-1}(\cQ_1(X)/\P(V_5))$ is smooth by Lemma~\ref{lemma:normality-f}.

Assume $k \le 1$.\
The above argument proves that $\cQ_1$ is 1-regular over $\P(V_5) \setminus (\Sigma_1(X) \cup \sY_{A,V_5}^{\ge 2})$.
Therefore, over $\P(V_5) \setminus (\Sigma_1(X) \cup \sY_{A,V_5}^{\ge 2})$, 
the scheme $F_{d-1}(\cQ_1(X)/\P(V_5))$ is smooth of the expected dimension.\ By Lemma~\ref{lemma:normality-f}, it remains to check that~\eqref{eq:dim-f-2} 
holds with $p = 1$ over $\sY^2_{A,V_5} \setminus \Sigma_1(X)$.\ Denote the map $F_{d-1}(\cQ_1(X)/\P(V_5)) \to \P(V_5)$ by $\varphi$.

Assume first $k = 0$.\ We  then have $n = 2d  + 1$ (hence $d = 1$ or $d = 2$) and $\cQ_1 \to \P(V_5)$ is 
(away from~$\Sigma_1(X)$) a fibration in quadrics of dimension $2d - 2$.\ Consequently, the dimensions of the fibers of~$F_{d-1}(\cQ_1(X)/\P(V_5))$ over the stratum $\sY^2_{A,V_5} \setminus \Sigma_1(X)$
are equal to~1 if $d = 1$, and to~2 if $d = 2$, while~$\dim(\sY^2_{A,V_5}) = 1$ by~\cite[Lemma~2.5]{DK2}.\ Thus, $\dim \varphi^{-1}(\sY^2_{A,V_5} \setminus \Sigma_1(X)) = d + 1$, while the right side of~\eqref{eq:dim-f-2}
is equal to $4 + d^2 - d(d+1)/2 - 2$, which is equal to $d + 1$ for $d \in \{1,2\}$.

Assume next $k = 1$.\ We then have  $n = 2d$ (hence $d = 2$ or $d = 3$) and 
$\cQ_1 \to \P(V_5)$ is (away from~$\Sigma_1(X)$) a fibration in quadrics of dimension $2d - 3$.\ Consequently, the dimensions of the fibers of~$F_{d-1}(\cQ_1(X)/\P(V_5))$ over the stratum $\sY^2_{A,V_5} \setminus \Sigma_1(X)$
are equal to~0 if $d = 2$, and to~1 if $d = 3$, hence $\dim \varphi^{-1}(\sY^2_{A,V_5} \setminus \Sigma_1(X)) = d - 1$,
while the right side of~\eqref{eq:dim-f-2} is equal to $4 + d(d-1) - d(d+1)/2 - 2$, 
which is equal to $d - 1$ for $d \in \{2,3\}$.
\end{proof}

For the reader's convenience, we summarize these results in a table.\ 
The first column is the dimension $n$ of the Gushel--Mukai variety $X$ (which must be ordinary when $n=3$).\ 
The second column indicates which Hilbert scheme $F_{d-1}(X)$ of linear subspaces contained in $X$ 
we consider (the superscript~$0$ in the last two lines means that we consider only some components, 
as detailed in \cite[Theorem~4.3]{DK2}) 
and the third column explains what is the associated double cover obtained as a Stein factorization (the scheme 
$\sY_{A,V_5}^{\ge \bullet} $ is the image of the Hilbert scheme morphism  $F_{d-1}(X) \to \P(V_5)$).

\begin{equation*}\renewcommand{\arraystretch}{1.5}
\begin{array}{c|c|c}
  \dim(X)& \textnormal{  Hilbert scheme morphism} &\textnormal{Double covering}   \\
\hline 
4& F_1(X) \to \P(V_5)&f_1 \colon \widetilde\sY_{A,V_5}^{\ge 1} \to \sY_{A,V_5}^{\ge 1}   \\
 5& F_1(X) \to \P(V_5)&f_0 \colon \widetilde\sY_{A,V_5}^{\ge 0} \to \sY_{A,V_5}^{\ge 0}  \\
  5& F^{0}_2(X) \to \P(V_5)&f_2 \colon \widetilde\sY_{A,V_5}^{\ge 2} \to \sY_{A,V_5}^{\ge 2}  \\
  6& F^{0}_2(X) \to \P(V_5) &f_1 \colon \widetilde\sY_{A,V_5}^{\ge 1} \to \sY_{A,V_5}^{\ge 1} \\
\end{array}
\end{equation*}

\subsection{EPW stratification of $\Gr(3,V_6)$}

We can also apply our results to the EPW stratification of the Grassmannian $\Gr(3,V_6)$ 
described by Iliev--Kapustka--Kapustka--Ranestad in \cite{IKKR}.\ We keep the assumption $\kk = \bC$.

As before, we consider  two Lagrangian subbundles in the trivial symplectic vector bundle 
\begin{equation*}
\cV := \bw3V_6 \otimes \cO_{\Gr(3,V_6)}.
\end{equation*}
The first is the trivial bundle ${\cA_1 := {}}A \otimes \cO_{\Gr(3,V_6)}$.
The second is the image ${\cA_2 := {}} V_6 \wedge \bw2\cU_3$ of the wedge product map
\begin{equation*}
V_6 \otimes \bw2\cU_3 \lra \bw3V_6 \otimes \cO_{\Gr(3,V_6)},
\end{equation*}
where $\cU_3$ is the tautological subbundle.\  It fits into an extension 
\begin{equation*}
0 \to \bw3\cU_3 \to V_6 \wedge \bw2\cU_3 \to (V_6/\cU_3) \otimes \bw2\cU_3 \to 0.
\end{equation*}
 Its fiber at a point $U_3 \in \Gr(3,V_6)$ is the subspace $V_6 \wedge \bw2U_3 \subset \bw3V_6$, hence it is indeed a Lagrangian subbundle of $\cV$.\ One can consider the Lagrangian intersection loci 
\begin{equation}\label{eq:def-zk}
\sZ_A^{\ge k} = S_k({\cA_1,\cA_2}) \subset \Gr(3,V_6)
\qquad\text{and}\qquad 
\sZ_A^k=\sZ_A^{\ge k}\setminus \sZ_A^{\ge k+1}.
\end{equation}
The results of Iliev--Kapustka--Kapustka--Ranestad that we need  can be summarized as follows (\cite[Proposition~2.6 and Corollary~2.10]{IKKR}).

\begin{theorem}\label{theorem:ikkr}
If the Lagrangian $A$ has no decomposable vectors, the following properties hold.
\begin{itemize}
\item[\textnormal{(a)}] $\sZ_A^{\ge 1} $ is an integral normal quartic hypersurface in $\Gr(3,V_6)$.
\item[\textnormal{(b)}] $\sZ_A^{\ge 2}$ is the singular locus of $\sZ_A^{\ge 1} $; 
it is an integral normal Cohen--Macaulay sixfold of degree~$480$.
\item[\textnormal{(c)}] $\sZ_A^{\ge 3}$ is the singular locus of $\sZ_A^{\ge 2} $; 
it is an integral normal Cohen--Macaulay threefold of degree~$4944$.
\item[\textnormal{(d)}] $\sZ_A^{\ge 4}$ is the singular locus of  $\sZ_A^{\ge 3}$; 
{it} is finite and smooth, and is empty for $A$ general.
\item[\textnormal{(e)}] $\sZ_A^{\ge 5}$ is empty.
\end{itemize}
\end{theorem}

In this situation, the line bundles $\cL$ and $\det (\cA_1) = \det(A \otimes \cO_{\Gr(3,V_6)})$  are both trivial, while 
\begin{equation*}
\det(\cA_2) = \det(V_6 \wedge \bw2\cU_3) \cong 
\det(\bw3\cU_3) \otimes \det((V_6/\cU_3) \otimes \bw2\cU_3) \cong
\cO_{\Gr(3,V_6)}(-4).
\end{equation*}
The line bundle $\cL^{\otimes(-10-k)} \otimes \det(\cA_1) \otimes \det(\cA_2) \cong \cO_{\Gr(3,V_6)}(-4)$ of Theorem~\ref{theorem:covering-lagrangian}  therefore has a unique square root, $\cO_{\Gr(3,V_6)}(-2)$.\ We always take for $\cM$ the restriction of $\cO_{\Gr(3,V_6)}(-2)$.\ Theorem~\ref{theorem:covering-lagrangian} 
gives the following result (the sheaves $\cR_k$ on $\sZ_A^{\ge k}$ were defined by~\eqref{eq:rk-ck-lag}).

\begin{theorem}\label{theorem:z-covers}
Assume that the Lagrangian $A$ has no decomposable vectors.
\begin{itemize}
\item[$0)$] There is a unique double cover $g_0 \colon \widetilde\sZ_A^{\ge 0} \to \Gr(3,V_6)$ branched over $\sZ_A^{\ge 1}$ such that 
\begin{equation*}
g_{0*}\cO_{\widetilde\sZ_A^{\ge 0}} \cong \cO_{\Gr(3,V_6)} \oplus \cO_{\Gr(3,V_6)}(-2).
\end{equation*}
The scheme $\widetilde\sZ_A^{\ge 0}$ is integral, normal, and smooth away from $g_0^{-1}(\sZ_A^{\ge 2})$.
\item[$1)$] 
There is a unique double cover $g_1 \colon \widetilde\sZ_A^{\ge 1} \to \sZ_A^{\ge 1}$ branched over $\sZ_A^{\ge 2}$ such that
\begin{equation*}
g_{1*}\cO_{\widetilde\sZ_A^{\ge 1}} \cong \cO_{\sZ_A^{\ge 1}} \oplus \cR_1(-2).
\end{equation*}
The scheme $\widetilde\sZ_A^{\ge 1}$ is  {integral}, normal, and smooth away from  $g_1^{-1}(\sZ_A^{\ge 3})$.
\item[$2)$] 
There is a unique double cover $g_2 \colon \widetilde\sZ_A^{\ge 2} \to \sZ_A^{\ge 2}$ branched over $\sZ_A^{\ge 3}$ such that
\begin{equation*}
g_{2*}\cO_{\widetilde\sZ_A^{\ge 2}} \cong \cO_{\sZ_A^{\ge 2}} \oplus \cR_2(-2).
\end{equation*}
The scheme $\widetilde\sZ_A^{\ge 2}$ is  {integral}, normal, and smooth away from  $g_2^{-1}(\sZ_A^{4})$.\ Moreover, $\cR_2 \cong \omega_{\sZ_A^{\ge 2}}(2)$.
\item[$3)$] 
There is a unique double cover $g_3 \colon \widetilde\sZ_A^{\ge 3} \to \sZ_A^{\ge 3}$ branched over $\sZ_A^{4}$ such that
\begin{equation*}
g_{3*}\cO_{\widetilde\sZ_A^{\ge 3}} \cong \cO_{\sZ_A^{\ge 3}} \oplus \cR_3(-2).
\end{equation*}
The scheme $\widetilde\sZ_A^{\ge 3}$ is  {integral}, normal, and smooth away from  $g_3^{-1}(\sZ_A^{4})$.
\end{itemize}
\end{theorem}
 
\begin{proof}
Repeat  the proof of Theorem~\ref{theorem:y-covers},  replacing Theorem~\ref{theorem:ogrady} with Theorem~\ref{theorem:ikkr}.
\end{proof}

 While the double covers in parts~1) and~3) are new,
the one in part~2) coincides with the hyperk\"ahler sixfold constructed in~\cite{IKKR} (and called the EPW cube)
under the assumption $\sZ_A^4 = \varnothing$.

\begin{lemma}
If $\sZ_A^4 = \varnothing$, the scheme $\widetilde\sZ_A^{\ge 2}$ is isomorphic to the EPW cube.
\end{lemma}

\begin{proof}
Denote the EPW cube by $g \colon \widetilde\sZ_A \to \sZ_A^{\ge 2}$.\
The fundamental group of $\widetilde\sZ_A$ is trivial since, by~\cite[Theorem~1.1]{IKKR}, $\widetilde\sZ_A$ is smooth 
and deformation equivalent to the Hilbert cube of a K3 surface.\  Since $\widetilde\sZ_A$ is smooth and $g^{-1}(\sZ_A^{\ge3})$ has codimension~3, 
the fundamental group of $\widetilde\sZ_A \setminus g^{-1}(\sZ_A^{\ge3})$ is trivial as well.\
Since 
\begin{equation*}
g \colon \widetilde\sZ_A \setminus g^{-1}(\sZ_A^{\ge3}) \lra \sZ_A^2
\end{equation*}
is an \'etale double covering by~\cite[Proposition~3.1]{IKKR},  the fundamental group of $\sZ_A^2$ is $\Z/2$.\

By~Theorem~\ref{theorem:z-covers}.2),
\begin{equation*}
g_2 \colon \widetilde\sZ_A^{\ge 2} \setminus g_2^{-1}(\sZ_A^{\ge3}) \lra \sZ_A^2
\end{equation*}
is also an \'etale double cover.\
Since also $\widetilde\sZ_A^{\ge 2}$ is integral, we have an isomorphism
\begin{equation*}
\widetilde\sZ_A^{\ge 2} \setminus g_2^{-1}(\sZ_A^{\ge3}) \cong \widetilde\sZ_A \setminus g^{-1}(\sZ_A^{\ge3})
\end{equation*}
of schemes over $\sZ_A^2$.\
Since $\widetilde\sZ_A^{\ge 2}$ and $\widetilde\sZ_A$ are both normal, they are isomorphic as schemes over $\sZ_A^{\ge 2}$.
\end{proof}

\begin{remark}
 {One can also relate the second quadratic fibration of a Gushel--Mukai variety (see~\cite[Section~4.4]{DK1})
  to the double covers $\widetilde\sZ_{A,V_5}^{\ge k} \to \sZ_{A,V_5}^{\ge k}$ obtained from the double covers
of Theorem~\ref{theorem:z-covers} by base change along the embedding $\Gr(3,V_5) \to \Gr(3,V_6)$.\ In this situation, an analogue of Lemma~\ref{lemma:q1-lag} is true (with the same proof, using~\cite[Proposition~4.10]{DK1}).\ It is hard however to control the normality of the schemes $\widetilde\sZ_{A,V_5}^{\ge k}$ and $\sZ_{A,V_5}^{\ge k}$,
so we do not know of an analogue of Corollary~\ref{corollary:q1-lag}.}
\end{remark}

\end{document}